\documentclass{article}
\def\Bbb{\mathbb}
\def\cal{\mathcal}
 \catcode`@=11
\usepackage{color}
\usepackage{amsthm,amssymb,amsfonts,mathrsfs}
\usepackage{amsmath}
\catcode`@=11 \@addtoreset{equation}{section}

\catcode`@=12
\newtheorem{Theorem}{Theorem}[section]
\newtheorem{Proposition}{Proposition}[section]
\newtheorem{Lemma}{Lemma}[section]
\newtheorem{Corollary}{Corollary}[section]
\theoremstyle{definition}

\newtheorem{Definition}{Definition}[section]
\newtheorem{Remark}{Remark}

\newtheorem{Assumptions}{Hypothesis}[section]

\def\cH{\mathcal H}

\def\ds{\displaystyle}

\def\fd{\mathfrak{d}}
\def\fg{\mathfrak{g}}
\def\fh{\mathfrak{h}}

\def\V1{\mathcal{V}}
\def\V2{\mathcal{V}_2}

\title{Null controllability for a degenerate population model in divergence form via Carleman estimates}
\author{{\sc Genni Fragnelli}\\
Dipartimento di Matematica\\ Universit\`{a} di Bari "Aldo Moro"\\
Via
E. Orabona 4\\ 70125 Bari - Italy\\ email: genni.fragnelli@uniba.it}

\date{}
\begin{document}

\maketitle

\vspace{0.3cm}

\centerline{ {\it  }}

\begin{abstract}
In this paper we consider a degenerate population equation in divergence form depending on time, on age and on space and we prove a related null controllability result via Carleman estimates.
\end{abstract}



\section{Introduction}

We consider the following population model in divergence form describing the dynamics of a single species:
\begin{equation} \label{1}
\begin{cases}
\displaystyle {\frac{\partial y}{\partial t}+\frac{\partial y}{\partial a}}
-(k(x)y_{x})_x+\mu(t, a, x)y =f(t,a,x)\chi_{\omega} & \quad \text{in } Q,\\
  y(t, a, 1)=y(t, a, 0)=0 & \quad \text{on }Q_{T,A},\\
 y(0, a, x)=y_0(a, x) &\quad \text{in }Q_{A,1},\\
 y(t, 0, x)=\int_0^A \beta (a, x)y (t, a, x) da  &\quad  \text{in } Q_{T,1}.
\end{cases}
\end{equation}
Here $Q:=(0,T)\times(0,A)\times(0,1)$, $Q_{T,A} := (0,T)\times (0,A)$, $Q_{A,1}:=(0,A)\times(0,1)$ and
$Q_{T,1}:=(0,T)\times(0,1)$. Moreover, $y(t,a,x)$  is the distribution of certain individuals  at location $x \in (0,1)$, at time $t\in(0,T)$, where $T$ is fixed, and of age $a\in (0,A)$. $A$ is the maximal age of life, while $\beta$ and $\mu$ are the natural fertility and the natural death rate, respectively. Thus, the formula $\int_0^A \beta y da$ denotes the distribution of newborn individuals at time $t$ and location $x$.  In the model $ \chi_\omega$ is the characteristic function of  the control region $\omega\subset(0,1)$; the function $k$ is the dispersion coefficient and we assume that it depends on the space variable $x$ and degenerates at the boundary of the state space.
We say that the function $k$ is 
\begin{Definition}\label{Wdef} {\bf Weakly degenerate (WD)}  if $k \in  W^{1,1}([0,1])$,
\[
 k>0 \text{ in }(0,1)  \text{ and } k(0)=k(1)=0,
 \]
and there exist $M_1, M_2\in (0,1)$ such that
$x k'(x)  \le M_1k(x) $ and $(x-1)k'(x) \le M_2 k(x)$ for all $x\in [0,1]$.
 \end{Definition} 

or 
\begin{Definition}\label{Sdef} {\bf Strongly degenerate (WD)}  if $k \in  W^{1,\infty}([0,1])$,
\[
 k>0 \text{ in }(0,1)  \text{ and } k(0)=k(1)=0,
 \]
and there exist $M_1, M_2\in [1,2)$ such that
$x k'(x)  \le M_1k(x) $ and $(x-1)k'(x) \le M_2 k(x)$ for all $x\in [0,1]$.
 \end{Definition} 
 For example, as $k$ one can consider $k(x)=x^\alpha(1-x)^\beta$, $\alpha, \beta >0$. Clearly, we say that $k$ is weakly or strongly degenerate only at $0$ if Definition \ref{Wdef} is satisfied only at $0$, i.e. $k \in  W^{1,1}([0,1])$, $
 k>0 \text{ in }(0,1], k(0)=0
$
and, there exists $M_1\in (0,1)$  or $M_1\in [1,2)$ such that
$x k'(x)  \le M_1k(x) $ for  all $x\in [0,1]$.  Analogously at $1$.\medskip

In the last centuries, population models have been widely investigated by many authors from many points of view (see, for example, \cite{diekmann}, \cite{diekmann1}, \cite{fmvM3}, \cite{ft}). From the general theory for the Lotka-McKendrick system, it is known that the asymptotic behavior of the solution depends on the so called net reproduction rate $R_0$: if $R_0>1$, the solution is exponentially growing; if $R_0<1$, the solution is exponentially decaying; if $R_0=1$, the solution tends to the steady state solution.   Clearly, if $R_0>1$ and the system represents the distribution of a damaging insect population or of a pest population, it is very worrying. For example, in $2017$ B. Zhong, C. Lv, W. Qin show that the net reproduction rate for the {\it Tirathaba rufivena} (which causes a lot of damages for the crop, for example, of fruits and flowers) depends on the temperature: it is $10.40$ if the temperature is $28^\circ C$ and it is $4.13$ is the temperature is $20^\circ C$ (see, for example, \cite{tirathaba}); in $2011$ S. S. Win, R. Muhamad, Z. A. M. Ahmad, N. A. Adam 
show that the net reproduction rate for the {\it Nilaparvata lugens} (which caused a lot of damages for the rice crop throughout South and South-East Asia since the early 1970's) was about $10$ (see, for example, \cite{nilaparvata}).
For this reason, recently great attention is given to null controllability. For example in \cite{he}, where \eqref{1} models an insect growth, the control corresponds to a removal of individuals by using pesticides.

There are a lot of papers that deal with null controllability for \eqref{1} when the dispersion coefficient {\it $k$ is a constant or a strictly positive function} (see, for example, \cite{An}).  If $y$ is independent of $a$ and $k$ degenerates at the boundary or at an interior point of the domain we refer, for example, to \cite{acf}, \cite{fm} and to \cite{fm2}, \cite{fm_hatvani}, \cite{fm_opuscola} if $\mu$ is singular at the same point of $k$.
To our best knowledge, \cite{aem} is the first paper where  $y$ depends on $t$, $a$ and $x$ and the dispersion coefficient $k$ can degenerate.  In particular, the authors assume that $k$ degenerates  at the boundary (for example $k(x) = x^\alpha,$  being $x \in (0,1)$ and $\alpha >0$). Using Carleman estimates for the adjoint problem, the authors prove null controllability for \eqref{1} under the condition $T\ge A$. However, this assumption is not realistic when $A$ is too large. To overcome this problem in \cite{em}, the authors used Carleman estimates and a fixed point method via the Leray - Schauder Theorem. However,  in \cite{em} the authors consider a dispersion coefficient that can degenerate only at a point of the boundary and they use
the fixed point technique in which
the birth rate $\beta$ must be in $C^2(Q)$ specially in the proof of \cite [Proposition 4.2]{em}.
In the recent paper \cite{fJMPA}, we studied null controllability for \eqref{1} in non divergence form and with a diffusion coefficient degenerating at a one point of the boundary domain or in an interior point. Observe that, in the case of a boundary degeneracy, we cannot derive the null controllability 
for \eqref{1} by the one of the problem in non  divergence form or vice versa, see \cite{cfr}. For this reason here we study the null controllability for \eqref{1} assuming that $k$  degenerates at the boundary of the domain and $T<A$ completing \cite{aem}. We underline that here, contrary to \cite{em} and \cite{fJMPA}, we assume also that $k$ can degenerate at both points of the boundary domain (see Theorem \ref{ultimo'}) and $\beta$ is only a continuous function. 
On the other hand, while in \cite{em}  the authors used Carleman estimates,  a
generalization of the Leray - Schauder fixed point Theorem and the multi-valued theory,
here we use only Carleman estimates , some results of \cite{fJMPA}
and a technique based on cut off functions, making the proof slimmer and easier to read.
Moreover, the technique that we use to prove Theorem \ref{ultimo'} can be applied also to the problem in non divergence form considered in \cite{fJMPA}, generalizing \cite[Theorem 4.8]{fJMPA}. Finally, in the proof of the last theorem we make precise a calculation of \cite[Theorem 4.8]{fJMPA} which was not accurate. 
Observe that in this paper, as in \cite{fJMPA}, we do not consider the positivity of the solution, even if it is clearly interesting. This problem is related to the minimum time, i.e. given $T$ cannot be very small, but it is  still a work in progress, see \cite{zuazua} for related results in non degenerate cases.

A final comment on the notation: by $c$ or $C$ we shall denote
{\em universal} strictly positive constants, which are allowed to vary from line to
line.

\section{Well posedness results}\label{sec3-1}
On the rates $\mu$ and $\beta$ we assume:
\begin{Assumptions}\label{ratesAss}
 The  functions $\mu$ and $\beta$ are such that
\begin{equation}\label{3}
\begin{aligned}
&\bullet \beta \in C(\bar Q_{A,1}) \text{ and } \beta \geq0  \text{ in } Q_{A,1}, \\
&\bullet \mu \in C(\bar Q) \text{ and }  \mu\geq0\text{ in } Q.
\end{aligned}
\end{equation}
\end{Assumptions}

To prove well posedness  of  \eqref{1}, we introduce, as in \cite{acf}, the following
 Hilbert spaces
\[
\begin{aligned}
 H^1_k:=\{ u \in L^2(0,1) \ \mid \ u \text{ absolutely
continuous in } [0,1],
\\  \sqrt{k} u_x \in  L^2(0,1) \text{ and } u(1)=u(0)=0 \}
\end{aligned}
\]
and
\[
H^2_k :=  \{ u \in H^1_k(0,1) |\,ku_x \in
H^1(0,1)\}. 
\]
 We have, as in \cite{acf} or \cite{fm1}, that the  operator
\[\mathcal A_0u:= (ku_{x})_x,\qquad    D(\mathcal A_0): = \cH^2_{k}(0,1)\]
is self--adjoint, nonpositive  and generates an analytic contraction
semigroup of angle $\pi/2$ on the space $L^2(0,1)$.

Now, setting $ \mathcal A_a u := \ds \frac{\partial  u}{\partial a}$, we have that
\[
\mathcal Au:= \mathcal A_a u - 
\mathcal A_0 u,
\]
for 
\[
u \in D(\mathcal A) =\left\{u \in L^2(0,A;D(\mathcal A_0)) : \frac{\partial u}{\partial a} \in  L^2(0,A;H^1_k(0,1)), u(0, x)= \int_0^A \beta(a, x) u(a, x) da\right\},
\]
generates a strongly continuous semigroup on $L^2(Q_{A,1}):= L^2(0,A; L^2(0,1))$ (see also \cite{iannelli}). Moreover, the operator $B(t)$ defined as
\[
B(t) u:= \mu(t,a,x) u,
\]
for $u \in D(\mathcal A)$, can be seen as a bounded perturbation of $\mathcal A$ (see, for example, \cite{acf}); thus also
$
(\mathcal A + B(t), D(\mathcal A))
$ generates a strongly continuous semigroup.

Setting $L^2(Q):= L^2(0,T;L^2(Q_{A,1}))$, the following well posedness  result holds:
\begin{Theorem}\label{theorem_existence}
Assume that $k$ is weakly or strongly degenerate at $0$ and/or at $1$.  For all $f \in
L^2(Q)$ and $y_0 \in L^2(Q_{A,1})$, the system \eqref{1} admits a unique solution 
\[y \in \mathcal U:= C\big([0,T];
L^2(Q_{A,1}))\big) \cap L^2 \big(0,T;
H^1(0,A; H^1_k(0,1))\big)\]
and 
\begin{equation}\label{stimau}
\begin{aligned}
\sup_{t \in [0,T]} \|y(t)\|^2_{L^2(Q_{A,1})} +\int_0^T\int_0^A\|\sqrt{k}y_x\|^2_{L^2(0,1)}dadt \le C \|y_0\|^2_{L^2(Q_{A,1})}  + C\|f\|^2_{L^2(Q)},
\end{aligned}
\end{equation}
where $C$ is a positive constant independent of $k, y_0$ and $f$.

In addition, if $f\equiv 0$,  then
$
y\in C^1\big([0,T];L^2(Q_{A,1})\big).
$ 
\end{Theorem}
For the existence of the solution and the regularity of it we refer, for example, to \cite{en} and  \cite{lm}. On the other hand, we postpone the proof of \eqref{stimau} to the Appendix.
\section{Carleman estimates}\label{sec-3}
From the general theory, it is known that null controllability
for a linear parabolic system
is, roughly speaking, equivalent to the observability for the associated homogeneous adjoint problem (see, for example, \cite{FeGu}). Thus, the key point is to prove such an  inequality.
 A usual strategy in showing the observability inequality  is to prove
that certain global Carleman estimates hold true for the adjoint  operator.
Hence, this section is devoted to obtain global
Carleman estimates for the operator which is the adjoint of the given one in both the
weakly and the strongly degenerate cases. In particular, we consider the following adjoint system associated to \eqref{1}:
\begin{equation} \label{adjoint}
\begin{cases}
\ds \frac{\partial z}{\partial t} + \frac{\partial z}{\partial a}
+(k(x)z_{x})_x-\mu(t, a, x)z =f ,& (t,a,x) \in Q,\\
  z(t, a, 0)=z(t, a, 1)=0, & (t,a) \in Q_{T,A},\\
   z(t,A,x)=0, & (t,x) \in Q_{T,1}.
\end{cases}
\end{equation}

\paragraph{Carleman inequalities when the degeneracy is at $0$.}\label{sec-3-1}

In this subsection we will consider the case when $k(0)=0$ and we assume that $\mu$ satisfies \eqref{3}. On the other hand, on $k$ we make additional assumptions:
\begin{Assumptions}\label{BAss01} The function
$k\in C^0[0,1]\bigcap C^1(0,1]$  is such that $k(0)=0$, $k>0$ on
$(0,1]$ and there exists $M_1\in (0,2)$ such that
$x k'(x)  \le M_1k(x) $ for all $x \in [0,1]$. Moreover, if $M_1\ge1$ one has to require that there exists $\theta \in (0,M_1]$, such that the function $x \ds \mapsto \frac{k(x)}{x^\theta}$ is nondecreasing near $0$.
\end{Assumptions}
We remark that the assumption $xa' \le M_1a$, with $M_1 < 2$,   is essential in all our results and it is the same made, for example, in \cite{acf}.  It  implies that
    $\ds \frac{1}{\sqrt{a}}\in L^{1}(0,1)$ and, in particular, if $K<1$, then
    $\ds \frac{1}{a} \in L^{1}(0,1)$. Thus, the case $M_1 \ge 2$ is excluded.  Summing up, we will confine our analysis to
the case of $\ds\frac{1}{\sqrt{a}} \in L^{1}(0,1)$. This is, however,
the interesting case from the viewpoint of null controllability.
In fact, if $\ds\frac{1}{\sqrt{a}} \notin L^{1}(0,1)$ and $y$ is independent of $a$, then
\eqref{1} fails to be null controllable on the whole interval
$[0,1]$, and regional null controllability is the only property
that can be expected, see \cite{cmv0}. 
\\
Moreover, the assumption ``$\exists \,\theta \text{ such
that} \text{ the function } x \rightarrow \ds \frac{k(x)}{x^{\theta}}
\mbox { is nondecreasing} \mbox { near } 0$'' is just  technical and it is equivalent to
the following one: ``$\exists \,\theta \text{ such that}\, \theta a
\le x a'$ near $0$''. Clearly, the prototype  is $k(x)= x^{\alpha}, \; \alpha \in (0, 2)$.

\vspace{0.5cm}

Now, let us introduce the weight function

\begin{equation}\label{13}
\varphi(t,a,x):=\Theta(t,a)(p(x) - 2 \|p\|_{L^\infty(0,1)}),
\end{equation}
where $\Theta$ is as in \eqref{571} and
$\displaystyle p(x):=\int_0^x\frac{y}{k(y)}dy$.
Observe that $ \varphi(t,a, x)  <0$ for all $(t,x) \in Q$ and
$\varphi(t,a, x)  \rightarrow - \infty \, \text{ as } t \rightarrow
0^+, T^-$  or  $a \rightarrow
0^+$.  The following estimate holds:

\begin{Theorem}\label{Cor1}
Assume that Hypothesis $\ref{BAss01}$ is satisfied. Then,
there exist two strictly positive constants $C$ and $s_0$ such that every
solution $v$ of \eqref{adjoint} in
\[
\mathcal{V}:=L^2\big(Q_{T,A}; H^2_k(0,1)\big) \cap H^1\big(0,T; H^1(0,A;H^1_k(0,1))\big)\]
satisfies, for all $s \ge s_0$,
\[
\begin{aligned}
\int_{Q}\left(s \Theta k v_x^2
                + s^3\Theta^3\text{\small$\displaystyle\frac{x^2}{k}$\normalsize}
                  v^2\right)e^{2s\varphi}dxdadt
&\le
C\int_{Q}f^{2}\text{\small$e^{2s\varphi}$\normalsize}dxdadt
\\&+ s\,C\int_0^T \int_0^A\Theta(t,a)\Big[ k v_x^2 e^{2s\varphi}\Big](t, a, 1)dadt.
\end{aligned}\]
\end{Theorem}

Clearly the previous  Carleman estimate holds
for every function $v$ that satisfies \eqref{adjoint} in $(0,T)\times(0,A)\times (B,C)$ as long as $(0,1)$ is substituted by $(B,C)$ and $k$ satisfies Hypothesis \ref{BAss01} in $(B,C)$.

\vspace{0.4cm}

\paragraph{Proof of Theorem \ref{Cor1}}
As a first step assume that $\mu \equiv 0$.

In order to prove Theorem \ref{Cor1},  we define, for
$s > 0$, the function
\[
w(t,a,x) := e^{s \varphi (t,a,x)}v(t,a,x)
\]
where $v$ is the solution of \eqref{adjoint} in $\mathcal{V}$; observe that,
since $v\in\mathcal{V}$, $w\in\mathcal{V}$. Clearly, one has that $w$ satisfies
\begin{equation}\label{1'}
\begin{cases}
(e^{-s\varphi}w)_t + (e^{-s\varphi}w)_a +(k (e^{-s\varphi}w)_x)_{x}  =f(t,a,x), & (t,x) \in
Q,
\\[5pt]
w(0, a, x)= w(T,a, x)= 0, & (a,x) \in Q_{A,1},
\\[5pt]
w(t,A,x)=w(t,0,x)=0, & (t,x) \in  Q_{T,1},
\\[5pt]
w(t, a,0)= w(t, a, 1)= 0, & (t,a) \in  Q_{T,A}.
\end{cases}
\end{equation}
Defining $Lw:= w_t + w_a+(kw_{x})_x$ and $L_sw:=
e^{s\varphi}L(e^{-s\varphi}w)$, the equation of \eqref{1'} can be
recast as follows
\begin{center}
$
L_sw =L^+_sw + L^-_sw=e^{s\varphi}f,
$
\end{center}
where
\[
\begin{cases}
L^+_sw :=( kw_{x})_x
 - s (\varphi_t+ \varphi_a) w + s^2k \varphi_x^2 w,
\\[5pt]
L^-_sw := w_t + w_a-2sk\varphi_x w_x -
 s(k\varphi_{x})_xw.
 \end{cases}
\]As usual,  we compute the inner product  $<L^+_sw, L^-_sw>_{L^2(Q)}$ whose first expression is given in the following lemma
\begin{Lemma}\label{lemma1}Assume Hypothesis $\ref{BAss01}$.
The following identity holds
\begin{equation}\label{D&BT}
\left.
\begin{aligned}
<L^+_sw,L^-_sw>_{L^2(Q)}
\;&=\;
\frac{s}{2} \int_Q(\varphi_{tt}+\varphi_{aa}) w^2dxdadt + s \int_Qk(x) (k(x)
\varphi_x)_{xx} w w_xdxdadt
\\&- 2s^2 \int_Qk \varphi_x \varphi_{tx}w^2dxdadt - 2s^2\int_{Q}k \varphi_x\varphi_{xa}w^2dxdadt\\
&+s
\int_Q(2 k^2\varphi_{xx} + kk'\varphi_x)w_x^2 dxdadt
+ s^3 \int_Q(2k \varphi_{xx} + k'\varphi_x)k
\varphi^2_x w^2dxdadt \\
&+s\int_{Q}\varphi_{at}w^2 dxdadt.
\end{aligned}\right\}\;\text{\{D.T.\}}
\end{equation}
\begin{equation}\nonumber
\hspace{55pt}
\text{\{B.T.\}}\;\left\{
\begin{aligned}
&
\int_{Q_{T,A}}[kw_xw_t]_{0}^{1} dadt+\int_{Q_{T,A}}\big[kw_xw_a\big]_0^{1}dadt -\frac{s}{2}\int_{Q_{A,1}} \left[\varphi_a w^2\right]_0^T dxda.\\& +
\int_{Q_{T,A}}[-s\varphi_x (k(x)w_x)^2 +s^2k(x)\varphi_t \varphi_x w^2 -
s^3 k^2\varphi_x^3w^2 ]_{0}^{1}dadt\\
& +
\int_{Q_{T,A}}[-sk(x)(k(x)\varphi_x)_xw w_x]_{0}^{1}dadt+ s^2 \int_{Q_{T,A}}\big[k\varphi_x\varphi_aw^2\big]_0^{1}dadt\\[3pt]&
-\frac{1}{2}\int_{Q_{T,1}} \big[kw_x^2\big]_0^Adxdt
+\frac{1}{2}\int_{Q_{T,1}}\big[ \big(s^2k \varphi_x^2 - s (\varphi_t+\varphi_a) \big)w^2\big]_0^Adxdt.\end{aligned}\right.
\end{equation}
\end{Lemma}
\vspace{5pt}
\begin{proof}
It results, integrating by parts,
\[
<L^+_sw,L^-_sw>_{L^2(Q)}= I_1+ I_2+ I_3+ I_4,
\]
where
\[
I_1= \int_{Q}(kw_{x})_x(w_t -2sk\varphi_x w_x -
 s(k\varphi_{x})_xw) dxdadt,
\]
\[
I_2=\int_{Q}\big(- s \varphi_t w
  +s^2k\varphi_x^2 w\big)(w_t -2sk\varphi_x w_x -
 s(k\varphi_{x})_xw) dxdadt,
\]
\[
I_3= \int_{Q}((kw_{x})_x-s (\varphi_t +\varphi_a)  w+ s^2k\varphi_x^2w)w_a dxdadt
\]
and
\[
I_4= -s\int_{Q}\varphi_a w  (w_t -2sk\varphi_x w_x -
 s(k\varphi_{x})_xw) dxdadt.
\]
By \cite[Lemma 3.1]{acf}, we get

\begin{equation}\label{!1}
\begin{aligned}
I_1+I_2&:=\frac{s}{2} \int_Q\varphi_{tt} w^2dxdadt + s \int_Qk(x) (k(x)
\varphi_x)_{xx} w w_xdxdadt
\\&- 2s^2 \int_Qk(x) \varphi_x \varphi_{tx}w^2dxdadt +s
\int_Q(2 k^2\varphi_{xx} + k(x)k'\varphi_x)w_x^2 dxdadt
\\&+ s^3 \int_Q(2k(x) \varphi_{xx} + k'\varphi_x)k(x)
\varphi^2_x w^2dxdadt+ \int_{Q_{T,A}}[k(x)w_xw_t]_{x=0}^{x=1} dadt\\& +
\int_{Q_{T,A}}[-s\varphi_x (k(x)w_x)^2 +s^2k(x)\varphi_t \varphi_x w^2 -
s^3 k^2\varphi_x^3w^2 ]_{x=0}^{x=1}dadt\\
& +
\int_{Q_{T,A}}[-sk(x)(k(x)\varphi_x)_xw w_x]_{x=0}^{x=1}dadt.
\end{aligned}
\end{equation}
Next, we compute $I_3$ and $I_4$. Integrating by parts, we have
\begin{equation}\label{!}
\begin{aligned}
I_3
&=-\frac{1}{2}\int_0^T\int_0^{1} \big[kw_x^2\big]_0^Adxdt
+\int_0^T\int_0^A\big[kw_xw_a\big]_0^{1}dadt \\
&+\frac{1}{2}\int_0^T\int_0^{1}\big[ \big(s^2k \varphi_x^2 - s (\varphi_t+\varphi_a) \big)w^2\big]_0^Adxdt
\\[3pt]
&
+
\frac{s}{2} \int_{Q} \varphi_{aa}w^2dxdadt+
\frac{s}{2} \int_{Q} \varphi_{ta}w^2dxdadt
- s^2\int_{Q}k \varphi_x\varphi_{xa}w^2dxdadt.
\end{aligned}
\end{equation}
On the other hand
\begin{equation}\label{!2}
\begin{aligned}
I_4&= \frac{s}{2}\int_{Q}\varphi_{at}w^2 dxdadt-
 s^2\int_{Q}k\varphi_x\varphi_{ax}w^2 dxdadt\\
 &-\frac{s}{2}\int_0^A\int_0^{1} \left[\varphi_a w^2\right]_0^T dxda+ s^2 \int_0^T\int_0^A\big[k\varphi_x\varphi_aw^2\big]_0^{1}dadt.
\end{aligned}
\end{equation}
Adding \eqref{!1} - \eqref{!2}, \eqref{D&BT} follows immediately.
\end{proof}
As a consequence of the definition of $\varphi$, one has the next estimate:
\begin{Lemma}\label{lemma2}Assume Hypothesis $\ref{BAss01}$. There exist two  strictly positive constants $C$ and $s_0$ such
that, for all $s\ge s_0$,  all solutions $w$ of \eqref{1'}
satisfy the following estimate
\[
sC\int_{Q}\Theta k w_x^2 dxdadt
+s^3C\int_{Q}\Theta^3 \frac{x^2}{k}w^2 dxdadt \le \big\{D.T.\big\} .
\]
\end{Lemma}
\begin{proof}
The distributed terms of
\:$<L^+_sw, L^-_sw>_{L^2(Q)}$ take the form
\begin{equation}\label{02}
\begin{aligned}
\big\{D.T.\big\}\; &= \frac{s}{2} \int_Q(\Theta_{tt}+\Theta_{aa})\Big(p- 2\|p\|_{L^{\infty}(0,1)}\Big) w^2dxdadt \\
&-2 s^2\int_{Q}\Theta{\Theta_t}\frac{x^2}{k}w^2dxdadt-2s^2 \int_{Q}\Theta \Theta_a\frac{x^2}{k}w^2dxdadt\\
& +s\int_Q \Theta(2k-k'x) w_x^2 dxdadt + s^3\int_Q\Theta^3 (2k-k'x)\frac{x^2}{k^2}w^2dxdadt\\
&+ s\int_Q\Theta_{ta}\Big(p- 2\|p\|_{L^{\infty}(0,1)}\Big)w^2dxdadt.
\end{aligned}
\end{equation}
Now, observe that 
 there exists $c>0$ such that 
 \begin{equation}\label{magtheta}
 \begin{aligned}
 &\Theta^{\mu} \le c \Theta ^\nu \mbox{ if } 0<\mu<\nu\\
 &|\Theta  \Theta_t|
\le c\Theta ^3, |\Theta {\Theta_a}|
\le c\Theta ^3, \\
&|{\Theta_{aa}}| \le c\Theta ^{\frac{3}{2}}, |{\Theta_{tt}}| \le c\Theta ^{\frac{3}{2}}
\text{ and  } |\Theta_{ta}|\le c\Theta ^{\frac{3}{2}}.
\end{aligned}
\end{equation}
Hence, proceeding as in the proof of \cite[Lemma 3.5]{acf}, one can deduce
\begin{equation}\label{terminenuovo0}
\begin{aligned}
\eqref{02} \ge &\frac{s}{2} \int_Q(\Theta_{tt}+\Theta_{aa})\Big(p- 2\|p\|_{L^{\infty}(0,1)}\Big) w^2dxdadt \\
&-s^3\frac{C}{4}\int_{Q}\Theta^3\frac{x^2}{k}w^2dxdadt-s^3\frac{C}{4} \int_{Q}\Theta^3 \frac{x^2}{k}w^2dxdadt\\
& +Cs\int_Q \Theta kw_x^2 dxdadt + s^3C\int_Q\Theta^3 \frac{x^2}{k}w^2dxdadt\\
&+ s\int_Q\Theta_{ta}\Big(p- 2\|p\|_{L^{\infty}(0,1)}\Big)w^2dxdadt.
\end{aligned}
\end{equation}
Now, it results
\[
s\left| \int_Q\Theta_{tt}\Big(p- 2\|p\|_{L^{\infty}(0,1)}\Big)\right|\le  \frac{C}{2}s \left|\int_Q\Theta^{3/2}
b(x)w^2 dxda dt\right| + s\frac{C}{2}\left|\int_Q\Theta ^{3/2}w^2 dxdadt\right|,
\]
where $b(x)= \int_0^x\frac{y}{k(y)}dy$. As in \cite[Lemma 3.5]{acf}, 
one can estimate the last two terms in the following way
\[
\frac{s}{2}\int_Q\Theta^{3/2} b(x)w^2 dxdadt
\le\frac{C}{16}s^3\int_Q\Theta^{3} \frac{x^2}{k(x)}w^2
dxdadt,
\]
for $s$ large enough and
 \[
 \begin{aligned}
\frac{s}{2}\left|\int_Q \Theta^{3/2} w^2 dxdadt \right|\le
\frac{C}{4}s\int_Q \Theta k(x) w_x^2 dxdadt + \frac{C}{16}s^3
\int_Q\Theta^3 \frac{x^2}{k(x)}w^2 dxdadt.
\end{aligned}
\]
Hence
 \[
 \begin{aligned}
s\left|  \int_Q\Theta_{tt}\psi(x)w^2dxdadt
\right| \le \frac{C}{4}s\int_0^1 \Theta k(x) w_x^2 dxdadt +
\frac{C}{8}s^3\int_Q\Theta^3 \frac{x^2}{k(x)}w^2
dxdadt.
\end{aligned}
\]
The same estimate holds also for
\[\displaystyle \int_Q\Theta_{aa}\Big(p- 2\|p\|_{L^{\infty}(0,1)}\Big)dxdadt
\quad \text{ and } \quad
\displaystyle \int_Q\Theta_{ta}\Big(p- 2\|p\|_{L^{\infty}(0,1)}\Big)dxdadt.\]
Using the above estimates in \eqref{terminenuovo0}
the thesis follows  immediately for $s_0$ large enough.
\end{proof}

The next lemma holds.
\begin{Lemma}\label{BT}
Assume Hypothesis $\ref{BAss01}$.
The boundary terms in \eqref{D&BT} become
\begin{equation}\label{BT1}
\begin{aligned}
\{B.T.\}=-\int_{Q_{T,A}}\!\![s\Theta xk w_x^2]_{0}^{1}dadt.
\end{aligned}
\end{equation}
\end{Lemma}
\begin{proof}Using the definition of $\varphi$, \cite[Lemma 3.6]{acf}, the boundary conditions of $w$ and proceeding as in \cite[Lemma 3.2]{fJMPA},
the boundary terms of  \:$<L^+_sw,
L^-_sw>_{L^2(Q)}$ become
\[
\begin{aligned}
 \big\{B.T.\big\}\;
&=
\int_{Q_{T,A}}\big[kw_xw_a\big]_0^{1}dadt -\frac{s}{2}\int_{Q_{A,1}} \left[\varphi_a w^2\right]_0^T dxda-
\int_{Q_{T,A}}\!\![s\Theta xk w_x^2]_{0}^{1}dadt\\
& + s^2 \int_{Q_{T,A}}\!\!\big[k\varphi_x\varphi_aw^2\big]_0^{1}dadt\
-\frac{1}{2}\int_{Q_{T,1}}\!\! \big[kw_x^2\big]_0^Adxdt
\\
&+\frac{1}{2}\int_{Q_{T,1}}\!\!\big[ \big(s^2k \varphi_x^2 - s (\varphi_t+\varphi_a) \big)w^2\big]_0^Adxdt\\
&=-\int_{Q_{T,A}}\!\![s\Theta xk w_x^2]_{0}^{1}dadt.
\end{aligned}
\]
\end{proof}

As a consequence of Lemmas \ref{BT} and \ref{lemma2}, we have
\begin{Proposition}\label{stima}Assume Hypothesis $\ref{BAss01}$.
There exist two  strictly positive constants $C$ and $s_0$ such
that, for all $s\ge s_0$,  all solutions $w$ of \eqref{1'} in $\mathcal{V}$ satisfy
\[
\begin{aligned}
&sC\int_{Q}\Theta k w_x^2 dxdadt
+s^3C\int_{Q}\Theta^3 \frac{x^2}{k}w^2 dxdadt
\\
&\le
C\left(\int_{Q}f^{2}\text{\small$e^{2s\varphi}$\normalsize}~dxdadt
+\int_{Q_{T,A}}\!\![s\Theta k w_x^2](t,a,1)dadt\right).
\end{aligned}
\]
\end{Proposition}

 Recalling the definition of $w$, we have $v= e^{-s\varphi}w$
and $v_{x}=  (w_{x}-s\varphi_{x}w)e^{-s\varphi}$. Thus, Theorem
\ref{Cor1} follows immediately by Proposition \ref{stima} when $\mu \equiv 0$.

\vspace{0.5cm}
Now, we assume that $\mu \not \equiv 0$.

To complete the proof of Theorem \ref{Cor1} we consider the function $\overline{f}=f+\mu v$.
 Hence,  there are  two  strictly positive constants $C$ and $s_0$ such that, for
all $s\geq s_0$, the following inequality holds
\begin{equation} \label{fati1?}
\begin{aligned}
\int_{Q}\left(s \Theta k v_x^2
                + s^3\Theta^3\text{\small$\displaystyle\frac{x^2}{k}$\normalsize}
                  v^2\right)e^{2s\varphi}dxdadt
&\le
C\int_{Q}\overline f^{2}\text{\small$e^{2s\varphi}$\normalsize}dxdadt
\\&+ s\,C\int_0^T \int_0^A\Theta(t,a)\Big[ k v_x^2 e^{2s\varphi}\Big](t, a, 1)dadt.
\end{aligned}
\end{equation}
On the other hand, we have
\begin{equation} \label{4'?}
\begin{aligned}
\int_{Q}\mid\overline{f} \mid^{2}e^{2s\varphi}\,dxdadt
\leq 2\Big(\int_{Q}|f|^{2}e^{2s\varphi}\,dxdadt
+\int_{Q}|\mu|^2|v|^{2}e^{2s\varphi}\,dxdadt\Big).
\end{aligned}
\end{equation}

Now, if $M_1<1$, applying the Hardy-Poincar\'{e} proved in \cite[Proposition 2.1]{acf} to the function $\nu:=e^{s\varphi}v$,
we obtain
 \begin{equation}\label{nu}
 \begin{aligned}
 \int_{Q}|\mu|^2|v|^{2}e^{2s\varphi}\,dxdadt&\le  \|\mu\|_{\infty}^2\int_{Q}
                    \nu^2dxdadt
\le C\int_{Q}
                 \nu^2   \frac{k(x)}{x^2}dxdadt 
     \le
    C\int_{Q}k(x)\nu^2_xdxdadt
    \\
   &\le
    C\int_{Q} k(x) e^{2s\varphi}v_x^2dxdadt
   + Cs^2\int_{Q}\Theta^2 e^{2s\varphi}\frac{x^2}{k} v^2dxdadt.
 \end{aligned}
  \end{equation}
If $M_1 \ge 1$, using the Young's inequality to the function $\nu:=e^{s\varphi}v$, we have
we obtain
 \begin{equation}\label{nu'}
 \begin{aligned}
 \int_{Q}|\mu|^2|v|^{2}e^{2s\varphi}\,dxdadt&\le  \|\mu\|_{\infty}^2\int_{Q}
                    \nu^2dxdadt\\
                    & \le C
\int_Q\left(\frac{k^{1/3}}{x^{2/3}}\nu^2\right)^{3/4}\left(
\frac{x^2}{k} \nu^2\right)^{1/4}dxdadt \\
&\le C\int_Q
\frac{k^{1/3}}{x^{2/3}}\nu^2dxdadt +
 C \int_Q
\frac{x^2}{k} \nu^2 dxdadt.
 \end{aligned}
  \end{equation}
  Now, consider the function $\gamma(x) = (k(x)|x^4)^{1/3}$. Clearly, $\displaystyle \gamma(x)=  k(x)
\left(\frac{x^2}{k(x)}\right)^{2/3}\le C k(x)$ and
$\displaystyle \frac{k^{1/3}}{x^{2/3}}=
\frac{\gamma(x)}{x^2}$. Moreover, using Hypothesis \ref{BAss01}, one
has that the function $\displaystyle\frac{\gamma(x)}{x^q} = \left(\frac{k(x)}{x^\theta}\right)^{\frac{1}{3}}$, where
$\displaystyle q: =\frac{4+\vartheta}{3}\in(1,2)$, is nondecreasing
near $0$.
The Hardy-Poincar\'{e} inequality (see \cite[Proposition 2.1.]{acf}) implies
\begin{equation}\label{hpappl}
\begin{aligned}
\int_Q\frac{k^{1/3}}{x^{2/3}}\nu^2dxdadt &= \int_Q
\frac{\gamma}{x^2} \nu^2 dxdadt \le C\int_Q  \gamma (\nu_x)^2 dxdadt \le C \int_Q k(x) \nu_x^2 dxdadt
\\
   &\le
    C\int_{Q} k(x) e^{2s\varphi}v_x^2dxdadt
   + Cs^2\int_{Q}\Theta^2 e^{2s\varphi}\frac{x^2}{k} v^2dxdadt.
\end{aligned}
\end{equation}
In any case, by \eqref{nu}, \eqref{nu'} and \eqref{hpappl}, \[
      \begin{aligned}
 \int_{Q}|\mu|^2|v|^{2}e^{2s\varphi}\,dxdadt
   &\le     C\int_{Q} k(x) e^{2s\varphi}v_x^2dxdadt
   + Cs^2\int_{Q}\Theta^2 e^{2s\varphi}\frac{x^2}{k} v^2dxdadt.
  \end{aligned}
  \]
   Using this last inequality in (\ref{4'?}), it follows
      \begin{equation}\label{fati2?}
      \begin{aligned}
      \int_{Q} |\bar{f}|^{2}\text{\small$~e^{2s\varphi}$\normalsize}~dxdadt
      &\le
      2\int_{Q} |f|^{2}\text{\small$~e^{2s\varphi}$\normalsize}~dxdadt
      +C\int_{Q}k(x) e^{2s\varphi} v_x^2 dxdadt
      \\&+ Cs^2\int_{Q} \Theta^2 e^{2s\varphi}\frac{x^2}{k} v^2dxdadt\\
      &\le  C\int_{Q} |f|^{2}\text{\small$~e^{2s\varphi}$\normalsize}~dxdadt
      +C\int_{Q}\Theta k(x) e^{2s\varphi} v_x^2 dxdadt
      \\&+ Cs^2\int_{Q} \Theta^3 e^{2s\varphi}\frac{x^2}{k} v^2dxdadt.
     \end{aligned} \end{equation}
   Substituting in
   (\ref{fati1?}), one can conclude

      \[
      \begin{aligned}
    &  \int_{Q}\left(s \Theta k v_x^2 +s^3\Theta^3\frac{x^2}{k} v^2\right)e^{2s\varphi}dxdadt
       \le
       C\Big(\int_{Q} |f|^{2}\text{\small$\frac{~e^{2s\varphi}}{k}$\normalsize}~dxdadt
        \\[3pt]& 
         + s\int_0^T \int_0^A\Theta(t,a)\Big[k v_x^2 e^{2s\varphi}\Big](t, a, 1)dadt \Big),
        \end{aligned}
      \]
for all $s$ large enough.\\

\paragraph{Carleman inequalities when the degeneracy is at $1$.}\label{sec-3-2}

In this subsection we will consider the case when $k(1)=0$. Again $\mu$ satisfies \eqref{3} and on $k$ we make the following assumption:
\begin{Assumptions}\label{BAss02} The function
$k\in C^0[0,1]\bigcap C^1[0,1)$  is such that $k(1)=0$, $k>0$ on
$[0,1)$ and there exists $M_2\in (0,2)$ such that
$(x-1) k'(x)  \le M_2k(x) $ for all $x \in [0,1]$. Moreover, if $M_2\le1$ one has to require that there exists $\theta \in (0,M_2]$, such that the function $x \ds \mapsto \frac{k(x)}{|1-x|^\theta}$ is nonincreasing near $1$.
\end{Assumptions}
For Hypothesis \ref{BAss02} we can make the same considerations made for Hypothesis \eqref{BAss01}.
\vspace{0.5cm}

As in the previous subsection, let us introduce the weight function
\begin{equation}\label{13'}
\bar\varphi(t,a,x):=\Theta(t,a)(\bar p(x) - 2 \|\bar p\|_{L^\infty(0,1)}),
\end{equation}
where $\Theta$ is as in \eqref{571} and
$
\displaystyle \bar p(x):=\int_{0}^x\frac{y-1}{k(y)}dy$. As before,
$\bar\varphi(t,a, x) <0$ for all $(t,x) \in Q$ and
$\bar  \varphi(t,a, x)  \rightarrow - \infty \, \text{ as } t \rightarrow
0^+, T^-$  or  $a \rightarrow
0^+$.  The following estimate holds:

\begin{Theorem}\label{Cor1'}
Assume that Hypothesis $\ref{BAss02}$ is satisfied. Then,
there exist two strictly  positive constants $C$ and $s_0$ such that every
solution $v$ of \eqref{adjoint} in $\mathcal V$
satisfies, for all $s \ge s_0$,
\[
\begin{aligned}
\int_{Q}\left(sk \Theta v_x^2
                  + s^3 \Theta^3\text{\small$\displaystyle\frac{(x-1)^2}{k}$\normalsize}
                    v^2\right)e^{2s\bar\varphi}dxdadt
&\le
C\int_{Q}f^{2}\text{\small$e^{2s\bar \varphi}$\normalsize}dxdadt
\\&+sC \int_0^T\int_0^A\!\!\! \Theta(t,a) \Big[ (1-x) v_x^2 e^{2s\bar\varphi}\Big](t, a,0)dadt.
\end{aligned}\]
\end{Theorem}

The previous  Carleman estimate holds
for every function $v$ that satisfies \eqref{adjoint} in  $(0,T)\times(0,A)\times(B,C)$ as long as $(0,1)$ is substituted by  $(B,C)$ and $k$ satisfies Hypothesis  \ref{BAss02} in $(B,C)$.

The proof of Theorem
\ref{Cor1'} is analogous to one of Theorem  \ref{Cor1} so we omit it. However, 
we underline that in the proof of Theorem \ref{Cor1} we use \cite[Lemma 3.1]{acf} which is proved only if $k$ degenerates at $0$; actually we observe that the proof of  \cite[Lemma 3.1]{acf} does not depend on the degeneracy point; hence, it holds also if $k(1)=0$. Instead, Lemma \ref{BT}, if $k(1)=0$, becomes
\[
-\int_{Q_{T,A}}\!\![s\Theta (x-1)k w_x^2]_{0}^{1}dadt= -\int_{Q_{T,A}}\!\![s\Theta k w_x^2](t,a,0)d adt;
\]
thus, Proposition \ref{stima} can be rewritten in the following way
\[
\begin{aligned}
&sC\int_{Q}\Theta k w_x^2 dxdadt
+s^3C\int_{Q}\Theta^3 \frac{(x-1)^2}{k}w^2 dxdadt
\\
&\le
C\left(\int_{Q}f^{2}\text{\small$e^{2s\bar\varphi}$\normalsize}~dxdadt
+\int_{Q_{T,A}}\!\![s\Theta k w_x^2](t,a,0)dadt\right).
\end{aligned}
\]
If $\mu \not \equiv 0$, we can 
proceed as in the proof of Theorem  \ref{Cor1}. However, while  in that case we use the Hardy-Poincar\'e inequality proved in \cite[Proposition 2.1]{acf} which holds only if $k(0)=0$, in this case we have to use the following inequality whose proof we postpone to the Appendix.
\begin{Proposition}[Hardy-Poincar\'{e} inequalities]\label{HP}
Assume that $k\, : \, [0,1] \longrightarrow \mathbb{R}_+$ is in
${\cal C}([0,1])$, $k(1)=0$, $k>0$ on $[0,1)$

{\bf Case (i):}

{\bf Hypothesis (HP1):} Assume that $k$ is such that there exists
$\theta \in (0,1)$ such that the function
$$
x \longrightarrow \dfrac{k(x)}{(1-x)^{\theta}} \mbox { is
nondecreasing in neighbourhood of } x=1 \,.
$$
\noindent Then, there is a constant $C>0$ such that for any
function $w$, locally absolutely continuous on $[0,1)$, continuous
at $1$ and satisfying

$$
w(1)=0 \,,\, \mbox{and } \int_0^1 k(x)|w^{\prime}(x)|^2 \,dx <
+\infty \,.
$$
\noindent the following inequality holds

\begin{equation}\label{hardy1}
\int_0^1 \dfrac{k(x)}{(1-x)^2}w^2(x)\, dx \leq C\, \int_0^1 k(x)
|w^{\prime}(x)|^2 \,dx \,.
\end{equation}
\noindent If hypothesis (HP1) is replaced by

{\bf Hypothesis (HP1)':} Assume that $k$ is such that there exists
$\theta \in (0,1)$ such that the function
$$
x \longrightarrow \dfrac{k(x)}{(1-x)^{\theta}} \mbox { is
nondecreasing in } [0,1)\,.
$$
\noindent Then, for any function $w$, locally absolutely
continuous on $[0,1)$, continuous at $1$ and satisfying

$$
w(1)=0 \,,\, \mbox{and } \int_0^1 k(x)|w^{\prime}(x)|^2 \,dx <
+\infty \,.
$$

\noindent the inequality (\ref{hardy1}) holds with the explicit
constant $C=\dfrac{4}{(1-\theta)^2}$.

{Case (ii):}

{\bf Hypothesis (HP2):} Assume that $k$ is such that there exists
$\theta \in (1,2)$ such that the function
$$
x \longrightarrow \dfrac{k(x)}{(1-x)^{\theta}} \mbox { is
nonincreasing in a neighbourhood of } x=1 \,.
$$
\noindent Then, there is a constant $C>0$ such that for any
function $w$, locally absolutely continuous on $[0,1)$ satisfying

$$
w(0)=0 \,,\, \mbox{and } \int_0^1 k(x)|w^{\prime}(x)|^2 \,dx <
+\infty \,.
$$

\noindent the inequality (\ref{hardy1}) holds.

If hypothesis (HP2) is replaced by

{\bf Hypothesis (HP2)':} Assume that $k$ is such that there exists
$\theta \in (1,2)$ such that the function
$$
x \longrightarrow \dfrac{k(x)}{(1-x)^{\theta}} \mbox { is
nonincreasing in } [0,1)\,.
$$
\noindent Then, for any function $w$, locally absolutely
continuous on $[0,1)$ satisfying

$$
w(0)=0 \,,\, \mbox{and } \int_0^1 k(x)|w^{\prime}(x)|^2 \,dx <
+\infty \,.
$$

\noindent the inequality (\ref{hardy1}) holds with the explicit
constant $C=\dfrac{4}{(1-\theta)^2}$.

\end{Proposition}

\vspace{0.5cm}

\section{Observability and controllability}\label{osservabilita}
In this section we will prove, as a
consequence of the Carleman estimates established in Section 3, 
observability inequalities  for the associated  adjoint problem of \eqref{1}. From now on, we assume that the control set $\omega$ is such that
\begin{equation}\label{omega}
\omega=  (\alpha, \rho)  \subset\subset  (0,1).
\end{equation}

Moreover, on $k$ and $\beta$ we assume the following assumptions:
\begin{Assumptions}\label{ipok}
The function
$k\in C^0[0,1]\bigcap C^1(0,1)$  is such that $k(0)=0=k(1)$, $k>0$ on
$(0,1)$ and there exist $M_1, M_2\in (0,2)$ such that
$x k'(x)  \le M_1k(x) $ and $(x-1)k'(x) \le M_2 k(x)$ for all $x\in [0,1]$.
Moreover, one has to require that:
\begin{enumerate} 
\item if $M_1\ge1$, there exists $\theta \in (0,M_1]$, such that the function $x \mapsto \frac{k(x)}{x^\theta}$ is nondecreasing near $0$;
\item  if $M_2\le1$, there exists $\gamma \in (0,M_2]$, such that the function $x \mapsto \frac{k(x)}{|1-x|^\gamma}$ is nonincreasing near $1$.
\end{enumerate}
\end{Assumptions}
\begin{Assumptions}\label{conditionbeta} Assume $T<A$ and suppose that there exists  
$\bar a \le T$
 such that
\begin{equation}\label{conditionbeta1}
\beta(a, x)=0 \;  \text{for all $(a, x) \in [0, \bar a]\times [0,1]$}.
\end{equation}
\end{Assumptions} 

Observe that Hypothesis \ref{conditionbeta} is the biological meaningful one. Indeed, $\bar a$ is the minimal age in which the female of the population become fertile, thus it is natural that before $\bar a$ there are no newborns. For other comments on  Hypothesis \ref{conditionbeta} we refer to \cite{fJMPA}.

\vspace{0,4cm}

Under the previous hypotheses, the following observability inequality holds:

\begin{Proposition}
\label{obser.}
Suppose that Hypotheses $\ref{BAss01}$ or $\ref{BAss02}$ or $\ref{ipok}$ and $\ref{conditionbeta}$ hold. Then,  for every $\delta \in (T,A)$, 
there exists a  strictly positive constant $C= C(\delta)$  such that  every
solution $v\in \mathcal U$ of \begin{equation}\label{h=0}
\begin{cases}
\ds \frac{\partial v}{\partial t} + \frac{\partial v}{\partial a}
+(k(x)v_{x})_x-\mu(t, a, x)v +\beta(a,x)v(t,0,x)=0,& (t,x,a) \in  Q,
\\[5pt]
v(t,a,0)=v(t,a,1) =0, &(t,a) \in Q_{T,A},\\
  v(T,a,x) = v_T(a,x) \in L^2(Q_{A,1}), &(a,x) \in Q_{A,1} \\
  v(t,A,x)=0, & (t,x) \in Q_{T,1},
\end{cases}
\end{equation}
satisfies
\begin{equation}\label{OI}
 \int_0^A\int_0^1 v^2(T-\bar a,a,x) dxda \le 
 C\left( \int_0^\delta \int_0^1v_T^2(a,x)dxda+ \int_0^T \int_0^A\int_ \omega v^2 dx dadt\right).
\end{equation}
Here $v_T(a,x)$ is such that $v_T(A,x)=0$ in $(0,1)$.
\end{Proposition}
\begin{Remark}
\begin{enumerate}
\item
If $T= \bar a$, the observability inequality given in the previous proposition is the corresponding of \cite[Proposition 3.1]{aem}, where the authors proved it under different assumptions and with $ T \ge A$.
\item
Moreover, as in \cite{fJMPA}, observe that in \eqref{OI}  the presence   of the integral $\ds \int_0^\delta \int_0^1 v_T^2(a,x)dxda$ is related to the presence of the term $\beta(a,x) v(t,0,x)$ in the equation of \eqref{h=0}. In fact,  estimating such a term using the method of characteristic lines, we obtain the previous integral. Obviously, if $v_T(a,x)=0$ a.e. in $(0,\delta) \times (0,1)$, we obtain the classical observability inequality.
\end{enumerate}
\end{Remark}

\vspace{0.3cm}

Before proving Proposition \ref{obser.} we will give some
results that will be very helpful. As a first step we
introduce the following class of functions
\[
{\cal W}:=\Big\{ v\;\text{solution of \eqref{h=0}}\;\big|\;v_T \in
D(\mathcal A^2)\Big\}, \] where 
$D(\mathcal A^2) = \Big\{u \in
D(\mathcal A) \;\big|\; \mathcal A u \in
D(\mathcal A)\;\Big\}
$
is densely
defined in $D({\cal A})$ (see, for example, \cite[Lemma 7.2]{b}) and
hence in $L^2(Q_{A,1})$. 
Obviously,
\[\begin{aligned}
{\cal W}= C^1\big([0,T]\:;D(\mathcal A)\big)&\subset \mathcal{V}:=L^2\big(Q_{T,A}; H^2_k(0,1)\big) \cap H^1\big(0,T; H^1(0,A;H^1_k(0,1))\big)
 \subset \cal{U}.
\end{aligned}\]

\begin{Proposition}[Caccioppoli's inequality]\label{caccio} 
Let $\omega'$ and $\omega$ two open subintervals of $(0,1)$ such
that $\omega ' \subset\subset \omega\subset\subset (0,1)$. Let $\psi(t,a,x):=\Theta(t,a)\Psi(x)$, where $\Theta$ is defined in \eqref{571}
and $\Psi\in C^1(0,1)$  is a strictly negative function.
Then, there exist two  strictly  positive constants $C$ and $s_0$ such that, for all $s \ge s_0$,
\begin{equation}\label{caccioeq}
\begin{aligned}
\int_{0}^T\int_0^A \int _{\omega'}   v_x^2e^{2s\psi } dxdadt
\ &\leq \ C\left( \int_{0}^T\int_0^A \int _{\omega}   v^2  dxdadt + \int_Q f^2 e^{2s\psi } dxdadt\right),
\end{aligned}
\end{equation}
for every
solution $v$ of \eqref{adjoint}.
\end{Proposition}
The proof of the previous proposition is similar to the one given in \cite{fJMPA}, but we repeat it in the Appendix  for the reader's convenience.

Moreover, the following non degenerate inequality proved in \cite{fJMPA} is crucial:
\begin{Theorem}\label{nondegenere}[see \cite[Theorem 3.2]{fJMPA}]
Let $z\in 
\mathcal{Z}$ be the solution of
\eqref{adjoint},
where $f \in L^{2}(Q)$, $k \in C^{1}([0,1])$ is a strictly
positive function and
\[\mathcal{Z}:=L^2\big(Q_{T,A}; H^2(0,1)\cap H^1_0(0,1)\big) \cap H^1\big(0,T; H^1(0,A;H^1_0(0,1))\big).\]  Then, there exist two strictly positive constants $C$ and $s_0$,
such that, for any $s\geq s_0$, $z$ satisfies the estimate
\begin{equation} \label{570'}
\begin{aligned}
&\int_{Q}(s^{3}\phi^{3}z^{2}+s\phi z_{x}^{2})e^{2s\Phi} dxdadt \leq C \Big(\int_{Q}f^{2}e^{2s\Phi}dxdadt
  -
s\kappa\int_0^T\int_0^A\left[ke^{2s\Phi}\phi(z_x)^2
\right]_{x=0}^{x=1}dadt\Big),
\end{aligned}
\end{equation}
where the functions $\phi$ and $\Phi$ are
defined as follows
\begin{equation}\label{571}
\begin{gathered}
\phi(t,a,x)=\Theta(t,a)e^{\kappa\sigma(x)}, \quad
\Theta(t, a)= \frac{1}{t^{4}(T-t)^{4}a^{4}},\\
\Phi(a,t,x)=\Theta(t,a)\Psi(x), \quad
\Psi(x)=e^{\kappa\sigma(x)}-e^{2\kappa\|\sigma\|_{\infty}},
\end{gathered}
\end{equation}
$(t,a,x)\in Q$, $\kappa>0$ and $\sigma (x) :=\mathfrak{d}\int_x^1\frac{1}{k(t)}dt$, where $\fd=\|k'\|_{L^\infty(0,1)}$.
\end{Theorem}

\begin{Remark}
The previous Theorem still holds under the weaker assumption $k \in W^{1, \infty}(0,1)$ without any additional assumption. 
\\
On the other hand, if we require $k \in W^{1,1}(0,1)$ then we have to add the following hypothesis:
{\it  there exist two functions $\fg \in L^1(0,1)$,
$\fh \in W^{1,\infty}(0,1)$ and two strictly positive constants
$\fg_0$, $\fh_0$ such that $\fg(x) \ge \fg_0$ and
\[-\frac{k'(x)}{2\sqrt{k(x)}}\left(\int_x^1\fg(t) dt + \fh_0 \right)+ \sqrt{k(x)}\fg(x) =\fh(x)\quad \text{for a.e.} \; x \in [0,1].\]}
\\
In this case, i.e. if $k \in W^{1,1}(0,1)$,  the function $\Psi$ in \eqref{571} becomes
\begin{equation}\label{Psi_new}
\Psi(x):= - r\left[\int_0^x
\frac{1}{\sqrt{k(t)}} \int_t^1
\fg(s) dsdt + \int_0^x \frac{\fh_0}{\sqrt{k(t)}}dt\right] -\mathfrak{c}, 
\end{equation}
where $r$ and $\mathfrak{c}$ are suitable strictly positive functions. For other comments on Theorem \ref{nondegenere} we refer to \cite{fJMPA}.
\end{Remark}

With the aid of Theorems \ref{Cor1}, \ref{Cor1'}, \ref{nondegenere} and Proposition \ref{caccio}, we can now show $\omega-$local Carleman estimates  for \eqref{adjoint}.

\begin{Theorem}\label{Cor2} Assume Hypothesis $\ref{BAss01}$. Then,
there exist two  strictly positive constants $C$ and $s_0$ such that every
solution $v$ of \eqref{adjoint} in
$
\mathcal {V}
$ 
satisfies, for all $s \ge s_0$,
\[
\begin{aligned}
\int_{Q}\left(s \Theta k  v_x^2
                + s^3\Theta^3\text{\small$\displaystyle \frac{x^2}{k}$\normalsize}
                  v^2\right)e^{2s\varphi}dxdadt
&\le
C\left(\int_{Q}f^{2}\text{\small$e^{2s\Phi}$\normalsize}~dxdadt+ \int_0^T \int_0^A\int_ \omega v^2 dx dadt\right).
\end{aligned}\]
\end{Theorem}
\begin{proof}
Let us consider a smooth function $\xi: [0,1] \to \Bbb R$ such that
  $$\begin{cases}
    0 \leq \xi (x)  \leq 1, &  \text{ for all } x \in [0,1], \\
    \xi (x) = 1 ,  &   x \in [0, (2\alpha +\rho)/3], \\
    \xi (x)=0, &     x \in [(\alpha +2\rho)/3,1].
    \end{cases}$$
    We define $w(t,a,x):= \xi(x)v(t,a, x)$ where $v\in\cal{V}$ satisfies \eqref{adjoint}.
    Then $w$   satisfies
    \[
    \begin{cases}
      w_t +w_a+(k  w_{x})_x- \mu w= \xi f + (k\xi_xv)_x+\xi_xkv_x=:h,&
      (t,a, x) \in Q,
      \\[5pt]
  w(t,a,0)=  w(t,a,1)=0, & (t,a) \in Q_{T,A}.
    \end{cases}
    \]
    Thus, applying Theorem \ref{Cor1} and Proposition \ref{caccio},
    \begin{equation}\label{add1}
    \begin{aligned}
&\int_0^T\int_0^A \int_0^{\frac{2\alpha+ \rho}{3}}\left(s \Theta k v_x^2
                + s^3\Theta^3\text{\small$\displaystyle\frac{x^2}{k}$\normalsize}
                  v^2\right)e^{2s\varphi}dxdadt \\
                  &=\int_0^T\int_0^A \int_0^{\frac{2\alpha+ \rho}{3}}\left(s \Theta k w_x^2
                + s^3\Theta^3\text{\small$\displaystyle\frac{x^2}{k}$\normalsize}
                  w^2\right)e^{2s\varphi}dxdadt\\
&\le
                  \int_{Q}\left(s \Theta k  w_x^2
                + s^3\Theta^3\text{\small$\displaystyle\frac{x^2}{k}$\normalsize}
                  w^2\right)e^{2s\varphi}dxdadt
\le
C\int_{Q}h^{2}\text{\small$e^{2s\varphi}$\normalsize}~dxdadt
\\&\le C\left( \int_{Q} f^2 e^{2s\varphi}dxdadt + \int_{0}^T\int_0^A \int _{\omega'}   v^2  dxdadt + \int_{0}^T\int_0^A \int _{\omega'}  v_x^2 e^{2s\varphi} dxdadt \right)
\\&
\le C \left( \int_{Q} f^2e^{2s\varphi} dxdadt+ \int_{0}^T\int_0^A \int _{\omega} v^2 dxdadt \right),
\end{aligned}
\end{equation}
where $\omega':= \ds\left(\frac{2\alpha +\rho}{3}, \frac{\alpha +2\rho}{3}\right)$.

Now, consider $z= \eta v$, where $\eta = 1-\xi$ and take $\bar \alpha \in (0, \alpha)$. Then $z$   satisfies
    \begin{equation}\label{problemz}
    \begin{cases}
      z_t +z_a+(k z_{x})_x- \mu z= \eta f +(k\eta_xv)_x+\eta_xkv_x=:h,&
      (t,a, x) \in Q_{T,A}\times (\bar \alpha, 1)=: \bar Q,
      \\[5pt]
  z(t,a,\bar\alpha)=  z(t,a,1)=0, & (t,a) \in Q_{T,A}.
    \end{cases}
    \end{equation}
    Clearly the equation satisfied by $z$ is not degenerate, thus applying Theorem \ref{nondegenere} and Proposition \ref{caccio}, one has
    \[
    \begin{aligned}
&\int_{\bar Q}(s^{3}\phi^{3}z^{2}+s\phi z_{x}^{2})e^{2s\Phi} dxdadt \leq C \int_{\bar Q}h^{2}e^{2s\Phi}dxdadt
\\
 &\le C\left( \int_{\bar Q}f^2 e^{2s\Phi}dxdadt + \int_0^T\int_0^A \int_{\omega'} (v^2 + v_x^2)e^{2s\Phi}dxdadt\right)\\
 &\le C \left( \int_Qf^{2}e^{2s\Phi}dxdadt+  \int_0^T\int_0^A \int_{\omega}v^2dxdadt\right).
\end{aligned}
    \]
    Hence
    \[ 
      \begin{aligned}
  &  \int_0^T\int_0^A \int_{\frac{\alpha+ 2\rho}{3}}^1   (s^{3}\phi^{3}v^{2}+s\phi v_{x}^{2})e^{2s\Phi} dxdadt=
 \int_0^T\int_0^A \int_{\frac{\alpha+ 2\rho}{3}}^1   (s^{3}\phi^{3}z^{2}+s\phi z_{x}^{2})e^{2s\Phi} dxdadt\\
 &\le C \left( \int_Qf^{2}e^{2s\Phi}dxdadt+  \int_0^T\int_0^A \int_{\omega} v^2dxdadt\right),
 \end{aligned}
    \]
    for a  strictly positive constant $C$.
Proceeding, for example, as in \cite{fm1}  one can prove the existence of $\varsigma>0$, such that,  for
all $(t,a,x)\in [0,T]\times[0,A]\times[\bar \alpha,1]$, we have
\begin{equation}\label{stimaphi}
e^{2s\varphi}\leq\varsigma e^{2s\Phi},
\frac{x^2}{k(x)}e^{2s\varphi}\leq\varsigma e^{2s\Phi}.
\end{equation}
Thus, for a  strictly  positive constant $C$,
   \begin{equation}\label{add2}
    \begin{aligned}
&\int_0^T\int_0^A \int_{\frac{\alpha+ 2\rho}{3}}^1 \left(s \Theta k v_x^2
                + s^3\Theta^3\frac{x^2}{k}
                  v^2\right)e^{2s\varphi}dxdadt \\
                  &\le C\left( \int_0^T\int_0^A \int_{\frac{\alpha+ 2\rho}{3}}^1   (s^{3}\phi^{3}v^{2}+s\phi v_{x}^{2})e^{2s\Phi} dxdadt\right)\\
                 &
\le C \left( \int_{Q} f^2 e^{2s\Phi} dxdadt+ \int_{0}^T\int_0^A \int _{\omega}  v^2  dxdadt \right).
\end{aligned}
\end{equation}
Now, consider $\tilde \alpha \in (\alpha, (2\alpha +\rho)/3)$, $\tilde \rho \in ((\alpha +2\rho)/3, \rho)$ and a smooth function $\tau: [0,1] \to \Bbb R$ such that
  $$\begin{cases}
    0 \leq \tau (x)  \leq 1, &  \text{ for all } x \in [0,1], \\
    \tau(x) = 1 ,  &   x \in [ (2\alpha +\rho)/3, (\alpha +2\rho)/3], \\
    \tau (x)=0, &     x \in [0, \tilde \alpha] \cup [\tilde \rho,1],
    \end{cases}$$
    and define $\zeta(t,a,x) := \tau(x)v(t,a,x)$. Clearly, $\zeta$ satisfies \eqref{problemz} with $h:= \tau f + (k \tau_{x} v)_x+ \tau_xkv_x$. Observe that in this case $\tau_ x, \tau_{xx} \not \equiv 0$ in $\bar\omega:= \ds\left(\tilde\alpha, \frac{2\alpha +\rho}{3}\right) \cup \left( \frac{\alpha +2\rho}{3}, \tilde\rho\right)$. As before, by Theorem \ref{nondegenere}, Proposition \ref{caccio} and \eqref{stimaphi}, we have
      \begin{equation}\label{add3}
    \begin{aligned}
&\int_0^T\int_0^A \int_{\frac{2\alpha+ \rho}{3}}^{\frac{\alpha+ 2\rho}{3}}\left(s \Theta k v_x^2
                + s^3\Theta^3\frac{x^2}{k}
                  v^2\right)e^{2s\varphi}dxdadt \\
                  &\le C\left( \int_0^T\int_0^A \int_{\frac{2\alpha+ \rho}{3}}^{\frac{\alpha+ 2\rho}{3}} (s^{3}\phi^{3}v^{2}+s\phi v_{x}^{2})e^{2s\Phi} dxdadt\right)\\
               &   = C\left(\int_0^T\int_0^A \int_{\frac{2\alpha+ \rho}{3}}^{\frac{\alpha+ 2\rho}{3}}   (s^{3}\phi^{3}\zeta^{2}+s\phi \zeta_{x}^{2})e^{2s\Phi} dxdadt\right)
                  \\
                 &
\le C \left( \int_{Q} f^2 e^{2s\Phi} dxdadt+ \int_{0}^T\int_0^A \int _{\omega}v^2 dxdadt \right).
\end{aligned}
\end{equation}
Adding \eqref{add1}, \eqref{add2} and \eqref{add3}, the thesis follows.
\end{proof}
Proceeding as before one can prove

\begin{Theorem}\label{Cor12} Assume Hypothesis $\ref{BAss02}$. Then,
there exist two  strictly positive constants $C$ and $s_0$ such that every
solution $v$ of \eqref{adjoint} in
$
\mathcal {V}
$
satisfies, for all $s \ge s_0$,
\begin{equation}\nonumber
\begin{aligned}
\int_{Q}\left(s \Theta k v_x^2
                + s^3\Theta^3\frac{(1-x)^2}{k}
                  v^2\right)e^{2s\varphi}dxdadt
\le
C\left(\int_{Q}f^{2}e^{2s\Phi}dxdadt+ \int_0^T \int_0^A\int_ \omega v^2 dx dadt\right).
\end{aligned}\end{equation}
\end{Theorem}

\begin{Remark}\label{remarkultimo}
Observe that the results of Theorems \ref{Cor2} and \ref{Cor12} still hold true if we substitute the domain $(0,T)\times (0,A)$ with a general domain $(T_1,T_2)\times (\delta,A)$, provided that $\mu$ and $\beta$ satisfy the required assumptions. In this case, in place of the function $\Theta$ defined in \eqref{571}, we have to consider the weight function
\[
\tilde \Theta(t,a):= \frac{1}{(t-T_1)^4 (T_2-t)^4(a-\delta)^4}.
\]
\end{Remark}
Using the previous local Carleman estimates one can prove the next observability inequalities.

\begin{Theorem}\label{Theorem4.4} Assume Hypothesis $\ref{BAss01}$ or \ref{BAss02} and  $\ref{conditionbeta}$ with $T > \bar a$. Then, for every $\delta \in (0,A)$, 
there exists a  strictly positive constant $C= C(\delta)$  such that every
solution $v$ of \eqref{h=0} in
$\mathcal V$
satisfies
\begin{equation}\nonumber
\begin{aligned}
 \int_0^A\int_0^1 v^2( T-\bar a,a,x) dxda  &\le C\int_0^{T} \int_0^\delta \int_0^1v^2(t,a,x) dxdadt\\
 &+ C\left( \int_0^{T}\int_0^1 v_T^2(a,x)dxda+ \int_0^T \int_0^A\int_ \omega v^2 dx dadt\right).
\end{aligned}\end{equation}
Moreover, if $v_T(a,x)=0$ for all $(a,x) \in (0,  T) \times (0,1)$, one has
\[
\begin{aligned}
 \int_0^A\!\!\int_0^1 v^2(T -\bar a,a,x) dxda  &\le C\left(\int_0^{T} \int_0^\delta\!\! \int_0^1v^2(t,a,x) dxdadt +\!\! \int_0^T\!\! \int_0^A\!\!\int_ \omega v^2 dx dadt\right).
\end{aligned}
\]
\end{Theorem}
\begin{proof} As in \cite{fJMPA} and using the method of characteristic lines, one can prove
the following implicit formula for $v$ solution of \eqref{h=0}:
\begin{equation}\label{implicitformula}
S(T-t) v_T(T+a-t, \cdot),
\end{equation}
if $t \ge  \tilde T + a$  and
\begin{equation}\label{implicitformula1}
v(t,a, \cdot)=\begin{cases}
S(T-t) v_T(T+a-t, \cdot)\!+\int_a^{T+a-t}S(s-a)\beta(s, \cdot)v(s+t-a, 0, \cdot) ds, &\Gamma\!= \!\bar a \\
\int_a^AS(s-a)\beta(s, \cdot)v(s+t-a, 0, \cdot) ds, & \Gamma\!= \!\Gamma_{A,T},
\end{cases}
\end{equation}
otherwise. Here  $(S(t))_{t \ge0}$ is the semigroup generated by the operator $\mathcal A_0 -\mu Id$ for all $u \in D(\mathcal A_0)$  ($Id$ is the identity operator), $\Gamma_{A,T}:= A -a +t-\tilde T$ and
\begin{equation}\label{Gamma}
\Gamma:= \min \{\bar a, \Gamma_{A,T}\}.
\end{equation}
In particular, it results
\begin{equation}\label{v(0)}
v(t,0, \cdot):= S(T-t) v_T(T-t, \cdot).
\end{equation}
Proceeding as in \cite[Theorem 4.4]{fJMPA}, with suitable changes, one has that there exists a positive constant $C$ such that:
\begin{equation}\label{t=011}
 \int_{Q_{A,1}} v^2(\tilde T,a,x) dxda  \le C\int_{\frac{T}{4}}^{\frac{3T}{4}} \int_{Q_{A,1}} v^2(t,a,x) dxdadt.
\end{equation}
Indeed, define, for $\varsigma >0$, the function $w= e^{\varsigma t }v$, where $v$ solves \eqref{h=0}.
Then $w$ satisfies
\begin{equation}\label{h=0'}
\begin{cases}
\ds \frac{\partial w}{\partial t} + \frac{\partial w}{\partial a}
+(k(x)w_{x})_x-(\mu(t, a, x)+ \varsigma) w =-\beta(a,x)w(t,0,x),& (t,x,a) \in  \tilde Q,
\\[5pt]
w(t,a,0)=w(t,a,1) =0, &(t,a) \in \tilde Q_{T,A},\\
  w(T,a,x) = e^{\varsigma T}v_T(a,x), &(a,x) \in Q_{A,1}, \\
  w(t,A,x)=0, & (t,x) \in \tilde Q_{T,1},
\end{cases}
\end{equation}
where $\tilde Q:= (\tilde T, T) \times Q_{A,1}$, $\tilde Q_{T,A}:= (\tilde T, T) \times (0,A)$ and $\tilde Q_{T,1}:= (\tilde T,T)\times (0,1)$.
Multiplying the equation of \eqref{h=0'} by $-w$ and integrating by parts on $Q_t:= (\tilde T,t) \times(0,A) \times(0,1)$, it results
\begin{equation}\label{3.56idriss}
\begin{aligned}
&-\frac{1}{2}\int_{Q_{A,1}} w^2(t,a,x) dxda + \frac{e^{\varsigma  \tilde T}}{2} \int_{Q_{A,1}}  v^2(\tilde T,a,x) dxda + \frac{1}{2} \int_{\tilde T}^t\int_0^1 w^2(\tau,0,x) dx d\tau\\
&+ \varsigma \int_{Q_t}w^2(\tau,a,x) dxdad\tau \le \int_{Q_t}\beta w(\tau,0,x)wdxdad\tau
\\
& \le \|\beta\|_{L^\infty(Q)}\frac{1}{\epsilon}\int_{Q_t}w^2dxdad\tau+ \epsilon A\|\beta\|_{L^\infty(Q)} \int_{\tilde T}^t\int_0^1 w^2(\tau,0,x)dxd\tau,
\end{aligned}
\end{equation}
for $\epsilon >0$. Choosing $\ds\epsilon= \frac{1}{2\|\beta\|_{L^\infty(Q)}A}$ and $\ds \varsigma =\frac{ \|\beta\|_{L^\infty(Q)}}{\epsilon}$, we have
\[
\begin{aligned}
 \int_{Q_{A,1}}  v^2(\tilde T,a,x) dxda  \le C\int_{Q_{A,1}}w^2(t,a,x) dxda  \le C \int_{Q_{A,1}} v^2(t,a,x) dxda.
\end{aligned}
\]
Then, integrating over $\ds \left[\frac{T}{4}, \frac{3T}{4} \right]$, we have \eqref{t=011}. 

Now, take $\delta \in (0, A)$. By \eqref{t=011}, we have
\begin{equation}\label{t=0}
\begin{aligned}
 \int_{Q_{A,1}} v^2(\tilde T,a,x) dxda  \le C \int_{\frac{T}{4}}^{\frac{3T}{4}} \left(\int_0^\delta + \int_\delta^A \right)\int_0^1 v^2(t,a,x) dxdadt.
\end{aligned}
\end{equation}
Consider the term $\ds\int_{\frac{T}{4}}^{\frac{3T}{4}}  \int_\delta^A\int_0^1v^2(t,a,x) dxdadt$. If Hypothesis \ref{BAss01} holds, proceeding as in the proof of Proposition \ref{Cor1} in the case $\mu\neq0$, one has 
\begin{equation}\label{terminenuovo11}
\begin{aligned}
 \int_0^1v^2dx
& 
\le C\left( \int_0^1 k v_x^2 dx +\int_0^1 \frac{x^2}{k} v^2dx\right) ,
\end{aligned}\end{equation}
for a  strictly positive constant $C.$  Hence,
\[
\begin{aligned}
\int_{\frac{T}{4}}^{\frac{3T}{4}}  \int_\delta^A\int_0^1v^2(t,a,x) dxdadt &\le C \int_{\frac{T}{4}}^{\frac{3T}{4}}  \int_\delta^A\int_0^1\tilde \Theta k v_x^2e^{2s\varphi} dxdadt \\
& +C \int_{\frac{T}{4}}^{\frac{3T}{4}}  \int_\delta^A\int_0^1 \tilde \Theta^3 \frac{x^2}{k}v^2e^{2s\varphi} dxdadt.
\end{aligned}
\]
Analogously, if Hypothesis \ref{BAss02} holds, then we obtain
\[
\begin{aligned}
\int_{\frac{T}{4}}^{\frac{3T}{4}}  \int_\delta^A\int_0^1v^2(t,a,x) dxdadt &\le C \int_{\frac{T}{4}}^{\frac{3T}{4}}  \int_\delta^A\int_0^1 \tilde\Theta k v_x^2e^{2s\varphi} dxdadt \\
& +C \int_{\frac{T}{4}}^{\frac{3T}{4}}  \int_\delta^A\int_0^1 \tilde\Theta^3 \frac{(1-x)^2}{k}v^2e^{2s\varphi} dxdadt.
\end{aligned}
\]
Thus, by Theorem \ref{Cor2} or \ref{Cor12},
\[
\int_{\frac{T}{4}}^{\frac{3T}{4}}   \int_\delta^A\int_0^1v^2(t,a,x) dxdadt \le
C\left(\int_{Q}f^2 dxdadt+ \int_0^T \int_0^A\int_ \omega v^2 dx dadt\right),
\]
where, in this case, $f(t,a,x):=-\beta(a,x)v(t,0,x)$.
Thus
\begin{equation}\label{t=01}
\int_{\frac{T}{4}}^{\frac{3T}{4}} \!\! \int_\delta^A\!\!\int_0^1v^2(t,a,x) dxdadt \le C\|\beta\|^2_{L^\infty(Q)}\!\left(\int_{Q}v^2(t,0,x)dxdadt+ \int_0^T \int_0^A\!\!\int_ \omega v^2 dx dadt\right),
\end{equation}
for a  strictly positive constant $C$. By \eqref{v(0)}, \eqref{t=01} and proceeding as in \cite{fJMPA}, we have
\begin{equation}\label{t=01'}
\int_{\frac{T}{4}}^{\frac{3T}{4}}\!\!  \int_\delta^A\!\!\int_0^1v^2(t,a,x) dxdadt \le C\|\beta\|^2_{L^\infty(Q)}\left(\int_{Q_{T,1}}\!\! v_T^2(a,x)dxda+ \int_0^T \!\!\int_0^A\!\!\int_ \omega v^2 dx dadt\right),
\end{equation}
for a  strictly positive constant $C$.
By \eqref{t=0} and \eqref{t=01'}, it results
\begin{equation}\label{t=03}
\begin{aligned}
 \int_{Q_{A,1}}  v^2(\tilde T,a,x) dxda  &\le C\int_0^T \int_0^\delta \int_0^1 v^2(t,a,x) dxdadt\\
 &+ C\left( \int_{Q_{T,1}} v_T^2(a,x)dxda+ \int_0^T \int_0^A\int_ \omega v^2dx dadt\right).
\end{aligned}
\end{equation}

\end{proof} 

\begin{Corollary}\label{CorOb} Assume $\bar a = T$, Hypotheses $\ref{BAss01}$ or $\ref{BAss02}$ and $\ref{conditionbeta}$.  Then, for every $\delta \in (0,A)$,
there exists a  strictly positive constant $C= C(\delta)$  such that every
solution $v$ of \eqref{h=0} in
$\mathcal V$
satisfies
\begin{equation}\nonumber
\begin{aligned}
 \int_0^A\int_0^1 v^2( 0,a,x) dxda  &\le C\int_0^T \int_0^\delta \int_0^1v^2(t,a,x) dxdadt\\
 &+ C\left( \int_0^T\int_0^1 v_T^2(a,x)dxda+ \int_0^T \int_0^A\int_ \omega v^2 dx dadt\right).
\end{aligned}\end{equation}
Moreover, if $v_T(a,x)=0$ for all $(a,x) \in (0, T) \times (0,1)$, one has
\[
\begin{aligned}
 \int_0^A\!\!\int_0^1  v^2(0,a,x) dxda  &\le C\left(\int_0^T \int_0^\delta\!\! \int_0^1 v^2(t,a,x) dxdadt +\!\! \int_0^T\!\! \int_0^A\!\!\int_ \omega v^2 dx dadt\right).
\end{aligned}
\]
\end{Corollary}
Actually, proceeding as in \cite{fJMPA} with suitable changes, we can improve the previous results in the following way:
\begin{Theorem}\label{CorOb1'}Assume Hypotheses $\ref{BAss01}$ or $\ref{BAss02}$ and $\ref{conditionbeta}$. Then, for every $\delta \in (T,A)$, 
there exists a  strictly positive constant $C= C(\delta)$  such that every
solution $v$ of \eqref{h=0} in
$\mathcal V$
satisfies
\begin{equation}\nonumber
 \int_0^A\int_0^1 v^2(T-\bar a,a,x) dxda \le 
 C\left( \int_0^\delta \int_0^1 v_T^2(a,x)dxda+ \int_0^T \int_0^A\int_ \omega v^2dx dadt\right).
\end{equation}
\end{Theorem}

\vspace{0.4cm}
By Theorem \ref{CorOb1'} and using a density argument, one can deduce
 Proposition \ref{obser.}. As a consequence one can prove, as in \cite{fJMPA}, the following null controllability results:

\begin{Theorem}\label{NCintermediate} Assume Hypotheses $\ref{BAss01}$ or $\ref{BAss02}$ or $\ref{ipok}$ and $\ref{conditionbeta}$. Then, given $T>0$ and $y_0 \in L^2(Q_{A,1})$,  for every $\delta \in (T,A)$,  there exists a control $f_\delta \in L^2(\tilde Q)$ such that the solution $y_\delta\in \mathcal U$ of 
\begin{equation} \label{1new}
\begin{cases}
\ds \frac{\partial y}{\partial t}  +\frac{\partial y}{\partial a}
-(k(x)y_{x})_x+\mu(t, a, x)y =f_\delta(t,x,a)\chi_{\omega}  &\quad \text{in } \tilde Q,\\
  y(t, a, 1)=y(t, a, 0)=0  &\quad \text{on } \tilde Q_{T,A},\\
 y(\tilde T , a, x)=y_0(a, x) &\quad \text{in }Q_{A,1},\\
 y(t, 0, x)=\int_0^A \beta (a, x)y (t, a, x) da  &\quad  \text{in } \tilde Q_{T,1},
\end{cases}
\end{equation}
satisfies
\[
y_\delta(T,a,x) =0 \quad \text{a.e. } (a,x) \in (\delta, A) \times (0,1).
\]
Moreover, there exists $C=C(\delta)>0$ such that
\begin{equation}\label{stimaf1}
\|f_\delta\|_{L^2(\tilde Q)} \le C \|y_0\|_{L^2(Q_{A,1})}.
\end{equation}
Here, we recall, $\tilde Q=(\tilde T,T)\times(0,A)\times(0,1)$, $\tilde Q_{T,A} = (\tilde T,T)\times (0,A)$
and $\tilde Q_{T,1}=(\tilde T,T)\times(0,1)$.
\end{Theorem}

Observe that if $T= \bar a$, Theorem  \ref{NCintermediate} is exactly the null controllability result that we expect. Indeed, in this case  \eqref{1new} coincide with \eqref{1}. On the other hand, if $T>\bar a$, the null controllability for \eqref{1} is given in the next theorem and it is based on the previous result:

\begin{Theorem}\label{ultimo} Assume Hypotheses $\ref{BAss01}$ or $\ref{BAss02}$  and $\ref{conditionbeta}$. Then, given $T \in (0, A)$ and $y_0 \in L^2(Q_{A,1})$,  for every $\delta \in (T,A)$,   there exists a control  $f_\delta \in L^2(Q)$ such that the solution $y_\delta$ of  \eqref{1}
satisfies
\[
y_\delta(T,a,x) =0 \quad \text{a.e. } (a,x) \in (\delta, A) \times (0,1).
\]
Moreover, there exists $C=C(\delta)>0$ such that
\begin{equation}\label{stimaf}
\|f_\delta\|_{L^2( Q)} \le C \|y_0\|_{L^2(Q_{A,1})}.
\end{equation}
\end{Theorem}
\begin{proof}
 The proof is similar to the one of \cite[Theorem 4.8]{fJMPA}. However, here we make all the calculations in order to make precise some steps in \cite[Theorem 4.8]{fJMPA}, where there is a misprint and a term was missing.

As a first step, set $\tilde T:= T-\bar a \in (0,T)$.
By Theorem \ref{theorem_existence}, there exists a unique solution $u$  of
\begin{equation} \label{1new2}
\begin{cases}
\ds \frac{\partial u}{\partial t}  +\frac{\partial u}{\partial a}
-(k(x)u_{x})_x+\mu(t, a, x)u =0  &\quad \text{in } (0, \tilde T) \times(0,A)\times (0,1),\\
  u(t, a, 1)=u(t, a, 0)=0  &\quad \text{on }  (0, \tilde T) \times(0,A),\\
 u(0 , a, x)=y_0(a, x) &\quad \text{in } (0,A)\times (0,1),\\
 u(t, 0, x)=\int_0^A \beta (a, x)u (t, a, x) da  &\quad  \text{in }(0, \tilde T) \times(0,1).
\end{cases}
\end{equation}
Set $\tilde y_0(a,x) :=u(\tilde T, a, x)$; clearly $\tilde y_0 \in L^2(Q_{A,1})$.
Now, consider 
\begin{equation} \label{1new1}
\begin{cases}
\ds \frac{\partial w}{\partial t}  +\frac{\partial w}{\partial a}
-(k(x)w_{x})_x+\mu(t, a, x)w =h(t,x,a)\chi_{\omega} & \quad \text{in } \tilde Q,\\
  w(t, a, 1)=w(t, a, 0)=0&  \quad \text{on } \tilde Q_{T,A},\\
 w(\tilde T , a, x)=\tilde y_0(a, x) &\quad \text{in }Q_{A,1},\\
 w(t, 0, x)=\int_0^A \beta (a, x)w (t, a, x) da & \quad  \text{in } \tilde Q_{T,1}.
\end{cases}
\end{equation}
Again, by Theorem \ref{theorem_existence}, there exists a unique solution $w_\delta $  of \eqref{1new1} and, by the previous Theorem, there exists a control $h_\delta \in L^2 (\tilde Q)$ such that
\[
w_\delta(T,a,x)=0 \text{ a.e. } (a,x) \in (\delta, A) \times(0,1)
\]
and
\[
\|h_\delta\|_{L^2(\tilde Q)} \le C \|\tilde y_0\|_{L^2(Q_{A,1})},
\]
for a positive constant $C$.

Now, define $y_\delta$ and $f_\delta$ by
\[
y_\delta:= \begin{cases} u, & \text{in}\; [0, \tilde T],\\
w_\delta, & \text{in} \; [\tilde T, T] \end{cases} \quad \text{and}\quad f_\delta:=\begin{cases}0, & \text{in}\; [0, \tilde T], \\ h_\delta, & \text{in}\; [\tilde T, T]. \end{cases}
\]
Then $y_\delta$ satisfies \eqref{1} and $f_\delta \in L^2(Q)$ is such that
\[
y_\delta(T,a,x)=0 \text{ a.e. } (a,x) \in (\delta, A) \times(0,1).
\]
Indeed $y_\delta(T,a,x) = w_\delta(T,a,x) =0$ a.e. $(a,x) \in (\delta, A) \times (0,1)$.

Now, we prove  \eqref{stimaf}.
As a first step, as in \cite{fJMPA}, we multiply the equation of \eqref{1new1} by $u$. Then, integrating over $Q_{A,1}$, we obtain:
\[
\begin{aligned}
\frac{1}{2}\frac{d}{dt} \int_0^A \int_0^1 u^2 dxda &+ \frac{1}{2}\int_0^1 u^2(t,A,x)dx + \int_0^A\int_0^1 ku_x^2 dxda + \int_0^A \int_0^1 \mu u^2 dxda \\
&=  \frac{1}{2}\int_0^1u^2(t,0,x)dx.
\end{aligned}
\] 
Hence, using the fact that $u(t,0,x)= \int_0^A \beta(a,x) u(t,a,x)da$, we have
\[
\frac{1}{2}\frac{d}{dt} \int_0^A \int_0^1u^2dxda \le \frac{1}{2}\int_0^1 \left(\int_0^A \beta(a,x) u(t,a,x)da\right)^2dx\le  \frac{C}{2}\int_0^A\int_0^1 u^2dxda.
\]
Setting $F(t):= \|u(t)\|^2_{L^2(Q_{A,1})}$ and multiplying the previous inequality by $e^{-Ct}$, it results
\[
\frac{d}{dt} \left(e^{-Ct}F(t)\right) \le 0.
\]
Integrating over $(0,t)$, for all $t\in [0,T]$, we obtain
\[
 \int_0^A \int_0^1u^2(t,a,x)dxda \le e^{CT}\int_0^A \int_0^1 u^2(0,a,x)dxda= e^{CT} \int_0^A \int_0^1 y_0^2(a,x) dxda.
\]
In particular,
\[
 \int_0^A \int_0^1u^2(\tilde T,a,x)dxda \le e^{CT} \int_0^A \int_0^1 y_0^2(a,x) dxda.
\]
Thus, 
\begin{equation}\label{ultima1}
\begin{aligned}
\|f_\delta \|_{L^2(Q)}^2 &= \int_{\tilde T}^T\int_0^A \int_0^1h_\delta^2dxdadt  \le C \|\tilde y_0\|_{L^2(Q_{A,1})}^2 \\
&=C \int_0^A \int_0^1 u^2(\tilde T, a,x)dxda \le C \int_0^A \int_0^1 y_0^2(a,x) dxda,
\end{aligned}
\end{equation}
for a  strictly positive constant $C$. 
Hence,
 \eqref{stimaf} follows.
\end{proof}

As a consequence of the previous theorem, we obtain the null controllability property if the coefficient $k$ degenerates at $0$ and at $1$ at the same time.

\begin{Theorem}\label{ultimo'} Assume Hypotheses $\ref{ipok}$  and $\ref{conditionbeta}$. Then, given $T \in (0, A)$ and $y_0 \in L^2(Q_{A,1})$,  for every $\delta \in (T,A)$,   there exists a control  $f_\delta \in L^2(Q)$ such that the solution $y_\delta$ of  \eqref{1}
satisfies
\[
y_\delta(T,a,x) =0 \quad \text{a.e. } (a,x) \in (\delta, A) \times (0,1).
\]
Moreover, there exists $C=C(\delta)>0$ such that
\begin{equation}\label{stimaf1}
\|f_\delta\|_{L^2( Q)} \le C \|y_0\|_{L^2(Q_{A,1})}.
\end{equation}
\end{Theorem}
\begin{proof}Fix $y_0 \in L^2(Q_{A,1})$ and consider the two problems
\begin{equation}\label{u1}
(P_1)\quad \begin{cases}
\displaystyle {\frac{\partial y}{\partial t}+\frac{\partial y}{\partial a}}
-(k(x)y_{x})_x+\mu(t, a, x)y =f(t,a,x)\chi_{\omega} & \quad \text{in } (0,T)\times(0,A)\times(0, \bar \beta),\\
  y(t, a, \bar \beta)=y(t, a, 0)=0 & \quad \text{on }Q_{T,A},\\
 y(0, a, x)=y_0(a, x) &\quad \text{in } (0,A)\times(0, \bar \beta),\\
 y(t, 0, x)=\int_0^A \beta (a, x)y (t, a, x) da  &\quad  \text{in } (0,T)\times(0, \bar \beta),
\end{cases}
\end{equation}
and
\begin{equation}\label{u2}
(P_2)\quad \begin{cases}
\displaystyle {\frac{\partial y}{\partial t}+\frac{\partial y}{\partial a}}
-(k(x)y_{x})_x+\mu(t, a, x)y =f(t,a,x)\chi_{\omega} & \quad \text{in } (0,T)\times(0,A)\times(\bar \alpha,1),\\
  y(t, a, 1)=y(t, a, \bar \alpha)=0 & \quad \text{on }Q_{T,A},\\
 y(0, a, x)=y_0(a, x) &\quad \text{in } (0,A)\times(\bar \alpha,1),\\
 y(t, 0, x)=\int_0^A \beta (a, x)y (t, a, x) da  &\quad  \text{in } (0,T)\times(\bar \alpha,1),
\end{cases}
\end{equation}
where $\bar \alpha \in (0, \alpha)$ and $\bar \beta\in (\beta,1)$. Thus, by Theorem \ref{ultimo},  there exist two controls $h_{1, \delta}$ and $h_{2, \delta}$ such that the solutions $u_{1,\delta}$ and $u_{2,\delta}$ of $(P_1)$ and $(P_2)$, associated to $h_{1, \delta}$ and $h_{2, \delta}$, respectively, satisfy
\[
u_{1, \delta}(T,a,x) =0 \quad \text{a.e. } (a,x) \in (\delta, A) \times (0,\bar \beta),
\]
and
\[
u_{2, \delta}(T,a,x) =0 \quad \text{a.e. } (a,x) \in (\delta, A) \times (\bar \alpha,1).
\]
Moreover, there exists $C>0$ such that
\[
\int_0^T\int_0^A\int_0^{\bar\beta}h_{1,\delta} \le C \|y_0\|_{L^2(Q_{A,1})}
\]
and
\[
\int_0^T\int_0^A\int_{\bar\alpha}^1h_{2,\delta} \le C \|y_0\|_{L^2(Q_{A,1})}
\]
Denote with $u_1$ and $h_1$ (respectively $u_2$ and $h_2$)  the trivial extensions of $u_{1, \delta}$ and $h_{1, \delta}$  (respectively $u_{2, \delta}$ and $h_{2, \delta}$) to $[\bar \beta,1]$ (respectively $[0, \bar \alpha]$), so that all functions are defined in  the interval $[0,1]$. Clearly, they depends always on $\delta$ and
\begin{equation}\label{stimaf2}
\|h_i\|_{L^2( Q)} \le C \|y_0\|_{L^2(Q_{A,1})}, \; i=1,2.
\end{equation}
Now, let $u_3$ be the solution of
\begin{equation}\label{u3}
\begin{cases}
\displaystyle {\frac{\partial y}{\partial t}+\frac{\partial y}{\partial a}}
-(k(x)y_{x})_x+\mu(t, a, x)y =0 & \quad \text{in } (0,T)\times(0,A)\times(0,1),\\
  y(t, a, 1)=y(t, a, 0)=0 & \quad \text{on }Q_{T,A},\\
 y(0, a, x)=y_0(a, x) &\quad \text{in } (0,A)\times(0,1),\\
 y(t, 0, x)=\int_0^A \beta (a, x)y (t, a, x) da  &\quad  \text{in } (0,T)\times(0,1),
\end{cases}
\end{equation}
and consider the three smooth cut off functions $\xi, \eta, \phi : [0,1] \to \Bbb R$ defined as
  $$\begin{cases}
    0 \leq \xi (x)  \leq 1, &  \text{ for all } x \in [0,1], \\
    \xi (x) = 1 ,  &   x \in [0,(2\alpha +\rho)/3], \\
    \xi (x)=0, &     x \in [(\alpha +2\rho)/3,1],
    \end{cases}$$
    $$\begin{cases}
    0 \leq \eta (x)  \leq 1, &  \text{ for all } x \in [0,1], \\
    \eta (x) = 0 ,  &   x \in [0, (2\alpha +\rho)/3], \\
    \eta (x)=1, &     x \in [(\alpha +2\rho)/3,1]
    \end{cases}$$
and $\phi:= 1-\xi-\eta$. Finally, take
\[
y(t,a,x) = \xi u_1+\eta u_2 + F(t) \phi u_3,
\]
where $F(t):= \ds \frac{T-t}{T}$. 

It is easy to verify that  $y(t,a,0)=y(t,a,1)=0$, $y(0,a,x) = y_0(a,x)$ (since $F(0)=1$) and
 $y(t, 0, x)=\int_0^A \beta (a, x)y (t, a, x) da$. Moreover,
 \[
 y(T,a,x)=0 \quad \text{a.e. } (a,x) \in (\delta, A) \times (0,1)
 \]
 and
  $y$ satisfies the equation of \eqref{1} with 
  \[
  \begin{aligned}
  f_\delta&= \xi h_1\chi_{\omega} +\eta h_2\chi_\omega -\frac{1}{T}\phi u_3 - F(t) k \phi'u_{3,x} - F(t)(k\phi'u_3)_x\\
  &- k \xi'u_{1,x} - (k\xi'u_1)_x-  k \eta'u_{2,x} -(k\eta'u_2)_x.
  \end{aligned}
  \]
  Obviously, the support of $f_\delta$ is contained in $\omega$ and, since $k \in C^1(\overline\omega)$, the terms  $(k\phi'u_3)_x$, $ (k\xi'u_1)_x$ and $(k\eta'u_2)_x$ are $L^2(0,1)$ (recall that $\phi'(x)=\xi'(x)=\eta'(x)=0$ for all $x \in (0,1) \setminus \omega$); thus  $f_\delta \in L^2(Q)$ as required. As in \cite{fm_opuscola}, estimate \eqref{stimaf1} follows by the definition of $f_\delta$, \eqref{stimaf2} and \eqref{stimau} for $u_i$, $i=1,2,3$.
\end{proof}
Observe that the previous result can be proved also for the problem in non divergence form considered in \cite{fJMPA}.
\section{Appendix} 
\subsection{Proof of \eqref{stimau}:}
Multiplying the equation of \eqref{1} by $y$ and  integrating over $(0,A) \times (0,1)$, we obtain
\[
\begin{aligned}
&\frac{1}{2}\frac{d}{dt}\|y(t)\|^2_{L^2(Q_{A,1})}+ \frac{1}{2}\int_0^1y^2(t,A,x)dx -\frac{1}{2}\int_0^1y^2(t,0,x)dx +\int_0^A\int_0^1ky_x^2dxda \\
&=-\int_0^A\int_0^1\mu y^2dxda + \int_0^A\int_\omega fy dx da.
\end{aligned}
\]
Hence, using the initial condition $y(t,0,x) = \int_0^A \beta(a,x) y(t,a,x) da$, the assumptions on $\beta$ and $\mu$ and the inequa\-li\-ty of  Jensen, one has
\begin{equation}\label{stima1}
\begin{aligned}
&\frac{1}{2}\frac{d}{dt}\|y(t)\|^2_{L^2(Q_{A,1})} + \frac{1}{2}\int_0^1y^2(t,A,x)dx  +\int_0^A\int_0^1ky_x^2dxda \\
&\le \frac{C}{2}\int_0^A\int_0^1y^2(t,a,x)dxda + \frac{1}{2} \int_0^A\int_0^1 f^2 dx da,
\end{aligned}
\end{equation}
where $C$ is a positive constant.
Since $\int_0^1y^2(t,A,x)dx $ and $\int_0^A\int_0^1ky_x^2dxda$ are positive, we deduce
\[
\frac{d}{dt}\|y(t)\|^2_{L^2(Q_{A,1})} \le C\|y(t)\|^2_{L^2(Q_{A,1})} + \|f(t)\|^2_{L^2(Q_{A,1})}.
\]
Setting $F(t):= \|y(t)\|^2_{L^2(Q_{A,1})} $ and multiplying the previous inequality by $e^{-Ct}$, one has
\begin{equation}\label{stima2}
\frac{d}{dt}(e^{-Ct}F(t)) \le  e^{-Ct}\|f(t)\|^2_{L^2(Q_{A,1})}.
\end{equation}
Integrating \eqref{stima2} over $(0,t)$, for all $t \in [0,T]$ it follows
\[
e^{-Ct}F(t) \le F(0) + \int_0^te^{-C\tau}\|f(\tau)\|^2_{L^2(Q_{A,1})}d\tau.
\]
Hence, for all $t \in [0,T]$,
\[
F(t) \le e^{CT}\left( F(0) + \int_0^T\|f(\tau)\|^2_{L^2(Q_{A,1})}d\tau\right)
\]
and
\begin{equation}\label{stima3}
\sup_{t \in [0,T] } \|y(t)\|^2_{L^2(Q_{A,1})} \le C\left(  \|y_0\|^2_{L^2(Q_{A,1})}  +\|f\|^2_{L^2(Q
)}d\tau\right).
\end{equation}
Therefore, by \eqref{stima1}, it follows
\[
\begin{aligned}
\frac{1}{2}\frac{d}{dt}\|y(t)\|^2_{L^2(Q_{A,1})} +\int_0^A\int_0^1ky_x^2dxda \le \frac{C}{2}\int_0^A\int_0^1y^2(t,a,x)dxda + \frac{1}{2} \int_0^A\int_0^1 f^2 dx da.
\end{aligned}
\]
Integrating over $(0,T)$, we have
\[
\begin{aligned}
\frac{1}{2}\|y(T)\|^2_{L^2(Q_{A,1})} +\int_0^T\int_0^A\int_0^1ky_x^2dxdadt&\le\frac{1}{2}\|y_0\|^2_{L^2(Q_{A,1})}+ \frac{C}{2}\int_0^T\int_0^A\int_0^1y^2(t,a,x)dxdadt\\
& + \frac{1}{2} \int_0^T\int_0^A\int_0^1 f^2 dx dadt.
\end{aligned}
\]
Hence, by \eqref{stima3},
\begin{equation}\label{stima4}
\begin{aligned}
\int_0^T\int_0^A\|\sqrt{k}y_x\|^2_{L^2(0,1)}dadt&\le\|y_0\|^2_{L^2(Q_{A,1})}+ C\int_0^T\|y(t)\|^2_{L^2(Q_{A,1})}dt+ \|f\|^2_{L^2(Q)}\\
&\le C\left(  \|y_0\|^2_{L^2(Q_{A,1})}  +\|f\|^2_{L^2(Q
)}d\tau\right)
\end{aligned}
\end{equation}
and \eqref{stimau} follows by \eqref{stima3} and \eqref{stima4}.
\subsection{Proof of Proposition \ref{HP}:}
We consider case (i), Hypothesis (HP1).
Fix $\beta \in (\theta,1)$ arbitrarily for the moment. Since
$w(1)=0$, we have
$$
\int_0^1 \dfrac{k(x)}{(1-x)^2}w^2(x)\, dx = \int_0^1 \dfrac{k(x)}{(1-x)^2}
\Big(\int_x^1 (1-y)^{\beta/2}w^{\prime}(y) (1-y)^{-\beta/2}\,dy
\Big)^2\, dx.
$$
\noindent This implies
$$
\int_0^1 \dfrac{k(x)}{(1-x)^2}w^2(x)\, dx \leq \int_0^1
\dfrac{k(x)}{(1-x)^2} \Big(\int_x^1 (1-y)^{\beta}|w^{\prime}(y)|^2 \,dy
\int_x^1(1- y)^{-\beta}\,dy \Big)\, dx.
$$

\noindent Hence, we have
$$
\int_0^1 \dfrac{k(x)}{(1-x)^2}w^2(x)\, dx \leq
\dfrac{1}{1-\beta}\int_0^1 \dfrac{k(x)}{(1-x)^{1+\beta}} \Big(\int_x^1
(1-y)^{\beta}|w^{\prime}(y)|^2 \,dy \Big)\, dx.
$$

\noindent By the Theorem of Fubini, it follows

\begin{equation}\label{1''}
\int_0^1 \dfrac{k(x)}{(1-x)^2}w^2(x)\, dx \leq
\dfrac{1}{1-\beta}\int_0^1 (1-y)^{\beta}|w^{\prime}(y)|^2
\Big(\int_0^y \dfrac{k(x)}{(1-x)^{1+\beta}} \,dx\Big)\, dy \,.
\end{equation}
Now, divide the right hand side of \eqref{1''} into three parts, i.e.
$$
\int_0^1 (1-y)^{\beta}|w^{\prime}(y)|^2 \Big(\int_0^y
\dfrac{k(x)}{(1-x)^{1+\beta}} \,dx\Big)\, dy =L_{\epsilon} +
M_{\epsilon} + N_{\epsilon},
$$
\noindent where
$$
L_{\epsilon}=\int_0^{1-\epsilon} (1-y)^{\beta}|w^{\prime}(y)|^2 \Big(\int_0^y
\dfrac{k(x)}{(1-x)^{1+\beta}} \,dx\Big)\, dy,
$$
$$
M_{\epsilon}= \int_{1-\epsilon}^1 (1-y)^{\beta}|w^{\prime}(y)|^2 \Big(\int_0^{1-\epsilon}
\dfrac{k(x)}{(1-x)^{1+\beta}} \,dx\Big)\, dy,
$$
\noindent and
$$
N_{\epsilon}= \int_{1-\epsilon}^1 (1-y)^{\beta}|w^{\prime}(y)|^2 \Big(\int_{1-\epsilon}^y
\dfrac{k(x)}{(1-x)^{1+\beta}} \,dx\Big)\, dy.
$$
\noindent Thanks to our hypothesis, there exists
$\epsilon>0$ such that the function
$$
x \longrightarrow \dfrac{k(x)}{(1-x)^{\theta}} \mbox{ is nondecreasing
on } [1-\epsilon,1);
$$
thus, for $N_{\epsilon}$, we have
\begin{equation}\label{7}
\begin{aligned}
N_{\epsilon}&= \int_{1-\epsilon}^1 (1-y)^{\beta}|w^{\prime}(y)|^2 \Big(\int_{1-\epsilon}^y
\dfrac{k(x)}{(1-x)^{1+\beta}} \,dx\Big)dy\\
& \leq \int_{1-\epsilon}^1 (1-y)^{\beta-\theta}|w^{\prime}(y)|^2
k(y)\left( \int_{1-\epsilon}^y(1-x)^{\theta -\beta -1}
dx\right)dy \leq\dfrac{1}{(\beta - \theta)}
\int_{1-\epsilon}^1 k(y) |w^{\prime}(y)|^2dy.
\end{aligned}
\end{equation}
\noindent For $M_{\epsilon}$, we have
\begin{equation}\label{8}
\begin{aligned}
M_{\epsilon}& \le\sup_{[0,1-\epsilon]}k  \int_{1-\epsilon}^1 (1-y)^{\beta}\frac{k(y)}{k(y)}|w^{\prime}(y)|^2\Big(\int_0^{1-\epsilon}
(1-x)^{-(1+\beta)} \,dx\Big)\, dy \\&\leq \dfrac{ \epsilon^{\theta-\beta}}{\beta k(1-\epsilon)}\sup_{[0,1-\epsilon]}k  \int_{1-\epsilon}^1(1-y)^{\beta-\theta}
k(y)|w^{\prime}(y)|^2\le
C  \int_{1-\epsilon}^1k(x) |w^{\prime}(x)|^2 \,dx.
\end{aligned}
\end{equation}
\noindent Proceeding in a similar way, we obtain
\begin{equation}\label{9}
L_{\epsilon} \leq C \int_0^{1-\epsilon}k(x)
|w^{\prime}(x)|^2 \,dx \,.
\end{equation}
\noindent Using (\ref{7}), (\ref{8}) and (\ref{9}) in (\ref{1''}),
we obtain
\begin{equation}\label{10'}
\int_0^1 \dfrac{k(x)}{(1-x)^2}w^2(x)\, dx \leq C\,\int_0^1 k(x)
|w^{\prime}(x)|^2 \,dx \,,
\end{equation}
\noindent where the constant $C$ depends on $a$, $\epsilon$,
$\theta$ and $\beta$. If one assumes that Hypothesis (HP1)' holds,
that is
$$
x \longrightarrow \dfrac{k(x)}{(1-x)^{\theta}} \mbox { is
nondecreasing on } [0,1),
$$
\noindent then, one can take $\epsilon=1$ in the above
computations, so that $L_{\epsilon}=M_{\epsilon}=0$. Using
then (\ref{7}) in (\ref{1''}), with $\epsilon=1$, one obtains
$$
\int_0^1 \dfrac{k(x)}{(1-x)^2}w^2(x)\, dx \leq
\dfrac{1}{(1-\beta)(\beta - \theta)} \int_0^1 k(x)
|w^{\prime}(x)|^2 \,dx \,.
$$
\noindent We then remark that this last estimate is optimal for
$\beta=\dfrac{\theta+1}{2}$, which gives the desired result.

We now consider case (ii), Hypothesis (HP2). Fix $\beta \in (1, \theta)$ arbitrarily
for the moment. As before
$$
\int_0^1 \dfrac{k(x)}{(1-x)^2}w^2(x)\, dx = \int_0^1 \dfrac{k(x)}{(1-x)^2}
\Big(\int_0^x (1-y)^{\beta/2}w^{\prime}(y) (1-y)^{-\beta/2}dy
\Big)^2dx;
$$
\noindent so that
$$
\int_0^1 \dfrac{k(x)}{(1-x)^2}w^2(x)\, dx \leq \int_0^1
\dfrac{k(x)}{(1-x)^2} \Big(\int_0^x (1-y)^{\beta}|w^{\prime}(y)|^2 \,dy
\int_0^x(1- y)^{-\beta}\,dy \Big)\, dx.
$$

\noindent It follows that
$$
\int_0^1 \dfrac{k(x)}{(1-x)^2}w^2(x)\, dx \leq \dfrac{1}{\beta
-1}\int_0^1 \dfrac{k(x)}{(1-x)^{1+\beta}} \Big(\int_0^x
(1-y)^{\beta}|w^{\prime}(y)|^2 \,dy \Big)\, dx.
$$

\noindent Applying the Theorem of Fubini, we have

\begin{equation}\label{2}
\int_0^1 \dfrac{k(x)}{(1-x)^2}w^2(x)\, dx \leq \dfrac{1}{\beta
-1}\int_0^1 (1-y)^{\beta}|w^{\prime}(y)|^2 \Big(\int_y^1
\dfrac{k(x)}{(1-x)^{1+\beta}} \,dx\Big)\, dy.
\end{equation}
\noindent As before, we rewrite
$$
\int_0^1 (1-y)^{\beta}|w^{\prime}(y)|^2 \Big(\int_y^1
\dfrac{k(x)}{(1-x)^{1+\beta}} \,dx\Big)\, dy= I_{\epsilon} +
J_{\epsilon}+ K_{\epsilon} \,,
$$
\noindent where
$$
I_{\epsilon}= \int_0^{1-\epsilon} (1-y)^{\beta}|w^{\prime}(y)|^2 \Big(\int_y^{1-\epsilon}
\dfrac{k(x)}{(1-x)^{1+\beta}} \,dx\Big)\, dy
$$
$$
J_{\epsilon}= \int_0^{1-\epsilon} (1-y)^{\beta}|w^{\prime}(y)|^2 \Big(\int_{1-\epsilon}^1
\dfrac{k(x)}{(1-x)^{1+\beta}} \,dx\Big)\, dy
$$
\noindent and
$$
K_{\epsilon}= \int_{1-\epsilon}^1 (1-y)^{\beta}|w^{\prime}(y)|^2 \Big(\int_y^1
\dfrac{k(x)}{(1-x)^{1+\beta}} \,dx\Big)\, dy.
$$
\noindent Thanks to our hypothesis, there exists $\epsilon>0$
such that the function
$$
x \longrightarrow \dfrac{k(x)}{(1-x)^{\theta}} \mbox { is
nonincreasing on } [1-\epsilon,1),
$$ thus $K_{\epsilon}$ can be estimated in the following way:
\begin{equation}\label{K}
\begin{aligned}
K_{\epsilon}&= \int_{1-\epsilon}^1 (1-y)^{\beta}|w^{\prime}(y)|^2 \Big(\int_y^1
\dfrac{k(x)}{(1-x)^{1+\beta}} \,dx\Big)\, dy\\&
 \leq
\int_{1-\epsilon}^1 (1-y)^{\beta}|w^{\prime}(y)|^2\dfrac{k(y)}{(1-y)^{\theta}} \int_y^1(1- x)^{\theta -\beta -1} \,dx \leq
 \dfrac{1}{(\theta - \beta)}
\int_{1-\epsilon}^1k(x) |w^{\prime}(x)|^2 \,dx.
\end{aligned}
\end{equation}
\noindent For $J_{\epsilon}$, we  can proceed in a similar way, obtaining
\begin{equation}\label{5}
\begin{aligned}
J_{\epsilon}&= \int_0^{1-\epsilon} (1-y)^{\beta}|w^{\prime}(y)|^2 \Big(\int_{1-\epsilon}^1
\dfrac{k(x)}{(1-x)^{1+\beta}} \,dx\Big)\, dy \leq
\epsilon^{-\beta}\dfrac{k(1-\epsilon)}
{\inf_{[0,1-\epsilon]}k}\int_0^{1-\epsilon}k(y)|w^{\prime}(y)|^2 \, dy\,.
\end{aligned}
\end{equation}
For $I_\epsilon$, we have
\begin{equation}\label{4''}
\begin{aligned}
I_{\epsilon}&= \int_0^{1-\epsilon} (1-y)^{\beta}|w^{\prime}(y)|^2\frac{k(y)}{k(y)} \Big(\int_y^{1-\epsilon}
\dfrac{k(x)}{(1-x)^{1+\beta}} \,dx\Big)\, dy \\
&\leq\frac{ \sup_{[0,1-\epsilon]}k}{\beta \inf_{[0,1-\epsilon]}k}
\int_0^{1-\epsilon}k(y)|w^{\prime}(y)|^2
(1-y)^{\beta}\epsilon^{-\beta} \,dy
\leq C
\int_0^{1-\epsilon} k(x) |w^{\prime}(x)|^2 \,dx \,.
\end{aligned}
\end{equation}
\noindent Using \eqref{K}, \eqref{5} and \eqref{4''} in (\ref{2}),
we obtain
\begin{equation}\label{6}
\int_0^1 \dfrac{k(x)}{(1-x)^2}w^2(x)\, dx \leq C\,\int_0^1 k(x)
|w^{\prime}(x)|^2 \,dx,
\end{equation}
\noindent where the constant $C$ depends on $a$, $\epsilon$,
$\theta$ and $\beta$. If one assumes that hypothesis (HP2)' holds,
that is
$$
x \longrightarrow \dfrac{k(x)}{(1-x)^{\theta}} \mbox { is
noninreasing on } [0,1),
$$
\noindent then, one can take $\epsilon=1$ in the above
computations, so that $J_{\epsilon}=I_{\epsilon}=0$. Using
then (\ref{3}) in (\ref{2}), with $\epsilon=1$, one obtains
$$
\int_0^1 \dfrac{k(x)}{(1-x)^2}w^2(x)\, dx \leq
\dfrac{1}{(\beta-1)(\theta - \beta)} \int_0^1 k(x)
|w^{\prime}(x)|^2 \,dx.
$$
\noindent We then remark that this last estimate is optimal for
$\beta=\dfrac{\theta+1}{2}$, which gives the desired result.

\subsection{Proof of Proposition \ref{caccio}:}
Let us consider a smooth function $\xi: [0,1] \to \Bbb R$ such that
  $$\begin{cases}
    0 \leq \xi (x)  \leq 1, &  \text{ for all } x \in [0,1], \\
    \xi (x) = 1 ,  &   x \in \omega', \\
    \xi (x)=0, &     x \in (0,1)\setminus\omega.
    \end{cases}$$
Then, integrating by parts one has
\begin{equation}\nonumber
\begin{aligned}
0
=&
\int_0^T\!\!\!\frac{d}{dt}\left(\int_0^A\int_0^1(\xi e^{s\psi})^2v^2dxda\right)dt
\\[3pt]
&=
\int_Q 2s\psi_t(\xi e^{s\psi})^2v^2 + 2(\xi e^{s\psi})^2v(-v_a-(kv_{x})_x+ \mu v+f) \:dxdadt
\\[3pt]
&=\; 2s\int_Q \psi_t(\xi e^{s\psi})^2v^2dxdadt+ 2s\int_Q \psi_a(\xi e^{s\psi})^2v^2dxdadt+ 2\int_Q
\left(\xi^2e^{2s\psi}\right)_xkvv_xdxdadt \\[3pt]
&+ 2\int_Q
(\xi^2e^{2s\psi}k)v_x^2dxdadt+2\int_Q
\xi^2e^{2s\psi}\mu v^2dxdadt +2\int_Q
\xi^2e^{2s\psi}fvdxdadt.
\end{aligned}
\end{equation}
Hence, using Young's inequality
\begin{equation}\nonumber
\begin{aligned}
2\int_Q \xi^2e^{2s\psi}kv_x^2dxdadt
&=
-2s\int_Q \psi_t\left(\xi e^{s\psi}\right)^2v^2dxdadt- 2s\int_Q \psi_a(\xi e^{s\psi})^2v^2dxdadt
\\[3pt]&-2\int_Q \left(\xi^2e^{2s\psi}\right)_xkvv_x\:dxdadt-2\int_Q
\xi^2e^{2s\psi}\mu v^2dxdadt\\
&-2\int_Q
\xi^2e^{2s\psi}fvdxdadt
\\[3pt]
&\le -2s\int_Q \psi_t(\xi e^{s\psi})^2v^2dxdadt - 2s\int_Q \psi_a(\xi e^{s\psi})^2v^2dxdadt\\
& + \int_Q\left( \sqrt{k} \frac{( \xi ^2e^{2s\varphi} )_x}{\xi
    e^{s\varphi} }v \right)^2dxdadt
   + \int_Q \xi ^2 e^{2s\varphi} k v_x ^2 dxdadt.
\\&+(2\|\mu\|_{L^\infty(Q)}+1)\int_Q
\xi^2v^2dxdadt +\int_Q
\xi^2e^{2s\psi}f^2dxdadt.
\end{aligned}
\end{equation}
Thus,
\begin{equation}\nonumber
\begin{aligned}
&\inf_{\omega'}\{k\}\int_0^T\int_0^A\int_{\omega'}e^{2s\psi}v_x^2dxdadt \\[3pt]&\le
\left(\sup_{\omega\times(0,T)}\Big\{\left| 4k\left(\xi
e^{s\psi}\:\right)_x^2-2s(\psi_t+\psi_a)(\xi
e^{s\psi})^2\right|\Big\}+(2\|\mu\|_{L^\infty(Q)}+1 \right)\int_0^T\int_0^A\int_{\omega}v^2dxdadt\\
&+\int_Q
f^2e^{2s\psi } dxdadt.
\end{aligned}
\end{equation}

\section*{Acknowledgments}

The author is a member of the Gruppo Nazionale per l'Analisi Matematica, la Probabilit\`a e le loro Applicazioni (GNAMPA) of the
Istituto Nazionale di Alta Matematica (INdAM) and she is supported by the FFABR ``Fondo
per il finanziamento delle attivit\`a base di ricerca'' 2017.

\end{document}